\documentclass[12pt]{amsart}

\usepackage{amsmath}
\usepackage{amssymb}
\usepackage{array}

\usepackage{verbatim} % has useful "comment" environment

\usepackage[colorlinks=true,linkcolor=blue,urlcolor=blue]{hyperref}

\usepackage{pstricks-add}

\theoremstyle{plain}
\newtheorem{theorem}{Theorem}[section]
\newtheorem{lemma}[theorem]{Lemma}

\newtheorem{corollary}[theorem]{Corollary}

\newtheorem{definition}[theorem]{Definition}
\theoremstyle{remark}
\newtheorem*{remark}{Remark}
\newtheorem*{example}{Example}
\newtheorem*{notation}{Notation}
\newtheorem*{acknowledgment}{Acknowledgment}

\numberwithin{equation}{section}

\newcommand{\bA}{\mathbb{A}}

\newcommand{\bB}{\mathbb{B}}

\newcommand{\K}{\mathbb{K}}
\newcommand{\bK}{\mathbb{K}}
\newcommand{\bP}{\mathbb{P}}

\newcommand{\R}{\mathbb{R}}

\newcommand{\N}{\mathbb{N}}
\newcommand{\Z}{\mathbb{Z}}
\newcommand{\PP}{\mathbb{P}}

\newcommand{\bF}{\mathbb{F}}

\newcommand{\mbA}{{\mathring \bA}}

\newcommand{\cX}{\mathcal{X}}
\newcommand{\cA}{\mathcal{A}}
\newcommand{\cY}{\mathcal{Y}}

\newcommand{\cD}{{\mathcal D}}
\newcommand{\cG}{{\mathcal G}}

\newcommand{\cU}{{\mathcal U}}

\newcommand{\cF}{\mathcal{F}}

\newcommand{\g}{\mathfrak{g}}

\newcommand{\h}{\mathfrak{h}}
\newcommand{\f}{\mathfrak{f}}
\newcommand{\e}{\mathfrak{e}}

\newcommand{\rr}{\mathbf{r}}
\newcommand{\bv}{\mathbf{v}}

\newcommand{\bfG}{\mathbf{G}}
\newcommand{\bfU}{\mathbf{U}}

\newcommand{\bfP}{\mathbf{P}}
\newcommand{\bfT}{\mathbf{T}}
\newcommand{\bfH}{\mathbf{H}}
\newcommand{\bfS}{\mathbf{S}}

\newcommand{\Gl}{\mathrm{GL}}
\newcommand{\Bij}{\mathrm{Bij}}

\newcommand{\SL}{\mathrm{SL}}
\newcommand{\Hom}{\mathrm{Hom}}

\newcommand{\Aut}{\mathbf{Aut}}

\newcommand{\End}{\mathrm{End}}

\newcommand{\Gras}{\mathrm{Gras}}
\newcommand{\ad}{\mathrm{ad}}

\newcommand{\id}{\mathrm{id}}

\newcommand{\pr}{\mathrm{pr}}

\newcommand{\Walg}{\mathrm{Walg}}

\newcommand{\Infaut}{\mathbf{Infaut}}

\newcommand{\set}{\mathrm{set}}

\newcommand{\La}{\Lambda}

\newcommand{\eps}{\varepsilon}

\newcommand{\inv}{^{-1}}
\newcommand{\msk}{\medskip}
\newcommand{\ssk}{\smallskip}
\newcommand{\nin}{\noindent}
\newcommand{\ul}{\underline}

\begin{document}

\title[Jordan Geometries -- an approach via Inversions]{Jordan Geometries -- an approach via inversions}
%\title[Jordan Geometries by Inversions]{Jordan Geometries by Inversions}

\author{Wolfgang Bertram}
\address{
Universit\'{e} de Lorraine, 
CNRS, Institut \'Elie Cartan de Lorraine, UMR 7502\\
Vandoeuvre-l\`es-Nancy, F-54506, France.}

\email{\url{wolfgang.bertram@univ-lorraine.fr}}



\begin{abstract}
{\em Jordan geometries} are defined as spaces $\cX$ equipped with point reflections
$J^{xz}_a$ depending on triples of points $(x,a,z)$, exchanging $x$ and $z$ and fixing $a$.
In a similar way, symmetric spaces have been defined by Loos (\cite{Lo69}) as spaces
equipped with point reflections $S_x$ fixing $x$, and therefore
the theories of Jordan geometries and of symmetric spaces are  closely related to 
each other --  in order to describe this link,
the notion of {\em inversive action} of torsors and of symmetric spaces is introduced.
Jordan geometries give rise both to inversive  actions  of certain abelian torsors
and of certain symmetric spaces, which in a sense are dual to each other.
By using the algebraic differential calculus dveloped in \cite{Be14},
we attach a tangent object to such geometries, namely a {\em Jordan pair}, resp.\ a {\em Jordan algebra}.
The present
approach works equally well over base rings in which $2$ is not invertible (and in particular over
$\Z$), and hence can be seen as a globalization of {\em quadratic Jordan pairs}; 
it also has a very transparent relation with
the theory of {\em associative geometries} from  \cite{BeKi1}.
\end{abstract}

\maketitle

\section*{Introduction}

Symmetries of order two -- called {\em reflections, inversions} or
{\em involutions}, according to  context  -- play a basic r\^ole in all of geometry, and some parts of geometry can 
be entirely reconstructed by using them (cf.\ the
``Aufbau der Geometrie aus dem Spiegelungsbegriff'', \cite{Ba73}).
In the present work, we will use the term ``inversion'' since the involutions  we use
can be interpreted as {\em (generalized) inverses in rings or algebras}:  geometrically,
the inversion map $x \mapsto x\inv$ in a unital associative algebra behaves like a {\em reflection
through a point}, with respect to the ``isolated'' fixed point $1$,  the unit element of the algebra. 
This choice of terminology should not lead
to conflict with the common one from Inversive Geometry, where the term ``inversion'' refers to 
reflections with respect to circles or spheres (cf.\ \cite{Wi}). 

\ssk
The inversion map %(and also the squaring map $x \mapsto x^2$) 
of an associative algebra
is a ``Jordan feature'', {\it i.e.}, it depends only on the symmetric part (``Jordan product'')
$x \bullet z = \frac{1}{2}(xz+zx)$ of the associative product, and it contains the whole information
of the Jordan product. 
The approach to Jordan algebras given in the book \cite{Sp} by T.\ Springer is based on this observation.
In the present work, we extend this approach to 
the {\em geometries} corresponding to Jordan algebraic structures.
We have defined such geometries, called {\em generalized projective geometries},
 in another way in \cite{Be02} -- the approach given there was not based on inversions, but
 rather on the various actions of a  scalar ring $\K$ on the geometry (a point of view 
 introduced by Loos in \cite{Lo79}); it relied in a crucial
way on {\em midpoints}, and thus on the existence of a scalar $\frac{1}{2}$ in $\K$.
The present approach does not have this drawback, and at the same time is simpler and more 
natural. Other advantages  are the  close relation with the {\em associative case}  studied in \cite{BeKi1,BeKi2}, and 
 a conceptual use of
``algebraic differential calculus'', keeping  close both to the  language of 
differential geometry and to the use of scalar extensions in algebraic geometry. 
Let us explain these items in more detail.
 
 \subsection{Jordan and associative structure maps} 
The general framework is given by a ``geometric space'' $\cX$ together with a 
{\em Jordan structure map} $J$, which associates to certain triples $(x,a,z)$ (called ``transversal'')
 a bijection $J^{xz}_a:\cX \to \cX$, subject to axioms that we call ``(geometric) Jordan identities''.
 Similarly,  an {\em associative structure map}  $M$ is given by
associating to certain (``closed transversal'') quadruples $(x,a,z,b)$ of points a bijection $M_{ab}^{xz}$ of $\cX$, such
that again certain axiomatic properties are satisfied. 
The precise form of these  properties is given in definitions \ref{def:Jordanstructure}  and \ref{def:assgeo}.
One of these properties is that 
$J^{xz}_a$ exchanges $x$ and $z$ and fixes $a$:
\begin{equation}\label{eqn:xaz} 
 J^{xz}_a(a)=a, \qquad J^{xz}_a(x)=z, \qquad J^{xz}_a(z)=x ,
 \end{equation}
and $M_{ab}^{xz}$ exchanges $a$ and $b$, as well as $x$ and $z$: 
\begin{equation}\label{eqn:xazb}
M_{ab}^{xz}(x)=z, \quad M_{ab}^{xz}(z)=x, \qquad M_{ab}^{xz}(a)=b, \qquad M_{ab}^{xz}(b)=a .
\end{equation}
 Instead of speaking of a family of maps, parametrized by certain tuples, we may also consider $J$ as a quaternary,
 and $M$ as a quintary {\em structure map} 
  \begin{equation}
 J : \cX^4 \supset \cD_3 \times \cX  \to \cX, \quad
 (x,a,z,y)\mapsto J^{xz}_a(y)  ,
 \end{equation}
  \begin{equation}
 M : \cX^5  \supset \cD_4' \times \cX   \to \cX, \quad
 (x,a,z,b,y)\mapsto M^{xz}_{ab} (y) ,
 \end{equation}
 where $\cD_n$ is the set of transversal,  and $\cD_n'$ the set of closed  transversal $n$-tuples.
The following example helps to get an idea  on the geometry   of such maps.

\subsection{An archetypical example: the projective line}
Let  $\cX = \bF \bP^1$ be  the projective line over a field $\bF$.
Here, $\cD_3'$ is just the set of  triples of pairwise different points from $\cX$.
Since
 the projective group $\bP \Gl(2,\bF)$ acts simply transitively on $\cD_3'$, for each triple $(x,a,z) \in \cD_3'$,
 there exists a unique projective map   $J^{xz}_a  \in \bP\Gl(2,\bF)$
 such that conditions (\ref{eqn:xaz}) hold.  It follows that $(J^{xz}_a)^2$ fixes all three points and hence is the identity; this 
justifies to  call $J^{xz}_a$ an {\em  inversion}.
The structure map $J$  has two interpretations,  an
   ``additive'' one ( A),  and a  ``multiplicative'' one  (M):

\msk \nin
{\bf (A)}
Choose $a = \infty$ (point at infinity). Then $\cU_a:=\cX \setminus \{ a \}$ is the affine line $\bF$, and
\begin{equation}\label{eqn:LineA}
J^{xz}_a(y)=  J^{xz}_\infty (y) = x- y +z 
\end{equation}
is a homography satisfying  (\ref{eqn:xaz}). 
This formula describes the {\em torsor structure} of the additive group $(\bF,+)$, that is, it is the ternary map
describing a ``group after forgetting its unit element'' (see Appendix \ref{App:SA}).
It ``works'' also if   $x=z$. 

\msk
\nin
{\bf (M)}
The multiplicative  interpretation comes  from the multiplicative  torsor structure $ay\inv b$
on  $(\bF^\times, \cdot)$ by letting $a=b$. Namely,
choose $x=\infty$, $z=0$; then $\cU_{xz}:= \cU_x \cap \cU_z = \bF^\times$, 
and  a homography satisfying (\ref{eqn:xaz}) is given by 
\begin{equation}
\label{eqn:LineM}
J^{xz}_a (y) =J^{0,\infty}_a(y) = ay\inv a  = a^2  y\inv.
\end{equation}

\ssk
The case of the projective line is ``special'' in the sense that the Jordan $J$-map comes from an associative
$M$-map: 
the  ``special''  feature, consequence of the simply transitive action of the projctive group on $\cD_3'$, is that,
if $(x,a,z)\in \cD_3'$ and $b$ is any point, there is a unique map
$M_{ab}^{xz} \in \bP \Gl(2,\bF)$ exchanging $x$ and $z$ and sending $a$ to $b$. 
Then this map must be  an involution (since 
 the square of a matrix $\bigl(\begin{smallmatrix}0& \lambda \cr 1  & 0 \end{smallmatrix}\bigr)$
is a multiple of the identity matrix), and hence also sends $b$ to $a$, that is, (\ref{eqn:xazb}) holds.
The structure map $M$ thus defined is a particular  instance of the one 
 studied, for general Grassmannians,  in \cite{BeKi1} (where the notation
$\Gamma(x,a,y,b,z)$ has been used for $M^{xz}_{ab} (y)$), and which define {\em associative geometries}.
In case of the projective line, choosing $(x,z)=(0,\infty)$, we see, in a similar way as above,  that
\begin{equation}\label{eqn:M1}
M_{ab}^{0,\infty}(y) = ay\inv b ,
\end{equation}
is a homography satisfying (\ref{eqn:xazb}). It  is related to the map $J$ defined by (\ref{eqn:LineM}) via
\begin{equation}
\label{eqn:JM}
J^{xz}_a = M^{xz}_{aa} .
\end{equation}
We say that a $J$-map is {\em special} if it comes from an $M$-map via this relation.
Note that this is the precise analog of defining a {\em special Jordan algebra} as one coming from
an associative algebra with product $ab$ by retaining the squaring operation $a^2$, which is the same as
restricting the associative product to the diagonal  $a=b$. 

\msk
For the case of the projective line, it is quite easy to obtain ``explicit formulae'' for the $J$- and $M$-maps,
that is, expressions as homographies where all arguments are ``generic''.  In a first step, in (\ref{eqn:M1}), 
we get for a generic value $a$ instead of $0$
\begin{equation}
M^{xz}_{\infty,a} (y) =  \frac{x-y+z - xa\inv z}{1 - a\inv y} . % x - (y-z) \frac{a-x}{a-y} 
\end{equation}
Indeed, the formula describes
 a homography exchanging $x$ and $z$ and sending $a$ to $\infty$, hence also $\infty$ to $a$.
As a corolloray, one has the nice formula
$M^{xz}_{a,\infty}(0)= x - xa\inv z + z$
(cf.\  \cite{BeKi1}, Prop.\ 1.7, for such formulae in  general Grassmannians). 
To get explicit formulae where all variables are generic,    observe
that the definitions of $J$ and of $M$ are ``natural'' in the sense that
\begin{equation}
\forall g \in \bP \Gl(2,\bF): \quad
J^{gx,gz}_{ga} (gy) = g J^{xz}_a(y), \quad
M^{gx,gz}_{ga,gb} (gy)= g M^{xz}_{ab}(y) .
\end{equation}
Now let $g(x)= \frac{x}{1 - a\inv x}$, a homography  sending $a$ to infinity, and use  (\ref{eqn:LineA}) to get
\begin{align}
J^{xz}_a(y) =
g\inv J^{gx,gz}_\infty(gy) & =
\frac{ \frac{x}{1-a\inv x}  - \frac{y}{1-a\inv y} + \frac{z}{1-a\inv z}}
{ 1 +  a\inv \bigl(  \frac{x}{1-a\inv x}  - \frac{y}{1-a\inv y} + \frac{z}{1-a\inv z} \bigr) }
\cr
& = \frac{ x - y +z - 2 xa\inv z + a^{-2} x y  z} {1 - 2 a\inv y + a^{-2} (xy+yz+xz)} .
\end{align}
Another formula for $J$, involving cross-ratios, can be obtained in a similar way from (\ref{eqn:LineM}),
and similarly for $M$.

\subsection{Jordan axioms}
In the general case, like in the preceding
 example, a Jordan structure map $J : \cD_3 \times \cX \to \cX$ has  two interpretations,
``additive'' (A) and ``multiplicative''  (M); moreover, there are   axioms of
{\em distributivity} (D) and {\em symmetry} (S) (compatibility).
The additive aspect of the structure map $J$ 
 deals with {\em abelian torsors}: 
for fixed $a\in \cX$, the partial law is the torsor structure underlying an abelian group (which is in fact an affine space)  $\cU_a$:
\begin{equation}
J^{xz}_a(y) = x - y + z,
\end{equation}
whereas (M) deals with (possibly non-abelian) {\em symmetric spaces}: for fixed $(x,z)$,
the partial law is a symmetric space structure on $\cU_{xz} = \cU_z \cap \cU_x$
\begin{equation}
J^{xz}_a(y) =: s_a(y) 
\end{equation}
where $s_a$ is the point reflection in $\cU_{xz}$ with respect to $a$. 
Distributivity (D) says that any of the bijections $J^{xz}_a$ is an automorphism (called {\em inner}) of the
whole structure $J$; finally, symmetry (S) means that for $x=z$, the symmetric space $\cU_{xx}$ coincides with the
abelian torsor $\cU_x$, seen as symmetric space:
\begin{equation}
J^{xx}_a = J^{aa}_x .
\end{equation}
Thus all ``Jordan axioms''  have a clear geometric meaning, and they arise in a natural way when 
merging the two structures ``abelian torsors'' and ``general symmetric spaces'' into  a single one. 

\subsection{Associative axioms}
As said above, a Jordan map $J$ is called {\em special} if it is related to an $M$-map via
(\ref{eqn:JM}).
An axiomatic definition of the ``associative structure map'' $M$ has been given in \cite{BeKi1}; in the present work, we 
give a slightly different definition (Section \ref{sec:ASM})
 by focusing on the invertible operators $M_{ab}^{xz}$ (whereas in \cite{BeKi1}
an algebraically more sophisticated axiomatics is used, which allows to deal also with non-invertible ``homotopes'' of
these operators). The ``special'' flavor of the associative case comes from the fact that, for the $M$-operators,
  both interpretations (A) and (M) are identical with  each other, dealing with possibly non-commutative torsors: this 
  follows from the ``strong  compatibility condition''  
 \begin{equation}
M_{ab}^{xz} = M_{xz}^{ab} .
\end{equation}

\subsection{Inversive actions}
For the deeper theory of the $J$- and $M$-maps, the dependence on their  {\em domains of definition} 
is  very important: 
in the example of the projective line, as well as in the general case, the argument $y$ of the bijections
$J^{xz}_a(y)$ and $M_{ab}^{xz}(y)$, may be {\em any} point of $\cX$.
Thus $\cU_a$, resp.\ $\cU_{xz}$,  is  not only a torsor, resp.\ a symmetric space, but at the same time 
comes with an {\em action on $\cX$ by inversions}, or shorter an {\em inversive action}.
Definition and basic properties of such actions are given in Appendix \ref{App:SA};
we have the impression that this notion might be useful also in general group theory, and especially in general
Lie theory.
Note that in \cite{BeKi1} it has been shown that the domain of the $M$-map can be further extended,
leading to quite subtle algebraic structures involving {\em semitorsors}; 
it remains an open problem whether a similar extension of domain of definition is possible for the $J$-map.

\subsection{Scalar action}
Associative or Jordan algebras are, by definition, defined over some base field or ring $\K$.
So far, this ring did not show up in the geometric setting -- put differently, one may say that
 the structures discussed so far are defined over $\Z$.
Indeed, one of our main motivations for this work was to develop a setting that can be defined over $\Z$, so that we may
postpone the use of action of scalars as long as possible; in contrast, the notion of {\em generalized projective geometry}
developed in \cite{Be02} depends from the very outset on such an action, or  {\em scaling map}
\begin{equation}
S: \K^\times \times \cX^3  \supset 
\K^\times \times \cD_2 \times \cX \to \cX, \quad (r,y,a,x) \mapsto S^r_{y,a} (x) =: r^a_y(x) .
\end{equation}
For the example of the projective line, $\K$ may be any  unital commutative subring of $\bF$, and then
$r^\infty_0(x) = rx$ is the usual multiple  $rx$ in $\bF$.
In the present work, the scaling map plays a less important r\^ole than in \cite{Be02}: it serves  only
as a conceptual framework allowing us to use {\em algebraic infinitesimal calculus}, see below.

 \subsection{The main examples}
 As we will explain below, to every Jordan algebraic or associative algebraic structure corresponds a Jordan, 
 resp.\ an associative geometry. We have the following examples, corresponding to the 
 main classes of such algebras:
 \begin{enumerate}
\item
 associative geometries are all given by the {\em Grassmannian geometries} introduced in \cite{BeKi1},
 possibly over non-commutative rings (section \ref{sec:ASM}),
% , and generalizing  the example of the projective line given above;
\item
 Jordan geometries come in four families:
 
 (2.1)  Grassmannian geometries, seen as Jordan geometries,
 
 (2.2)  geometries of Lagrangian subspaces  of a quadratic form,
 
 (2.3)  geometries of Lagrangian subspaces of a symplectic form,
 
 (2.4)  projective quadrics (defined in Subsection \ref{sssec:Quadrics}),
 
 \item 
two kinds of exceptional Jordan geometries related to the octonions.
 \end{enumerate}
 
\nin The Lagrangian Grassmannians are subgeometries of the associative Grassmannian geometries, fixed under
orthocomplementation, which is an  anti-automorphism of order two
(see \cite{BeKi2}), and the exceptional geometries are related to the  {\em Mou\-fang torsors} 
studied in \cite{BeKi12}.

\subsection{Unit elements, idempotents, and self-duality}
Existence of {\em unit elements} in algebras (Jordan or associative) corresponds
to {\em self-duality of geometries}, meaning that a geometry is {\em canonically} isomorphic to
its dual geometry (see \cite{Be03}).
 For instance, the {\em projective line} is self-dual (canonically isomorphic to its dual projective line!), 
 whereas  higher dimensional projective spaces are not.
 % ({\em choice} of some orthocomplementation). 
In the  setting of Jordan geometries, a self-dual geometry may be characterized by the existence of 
{\em  closed transversal triples}  $(a,b,c)$.  Then the inversions $J^{ab}_c$, $J^{ba}_c$,
$J^{ac}_b$ are all defined and generate a permutation group $S_3$; if we add 
$J^{aa}_b$  to the set of generators, they generate a group which is a homomorphic image of
$\PP\Gl(2,\Z)$ (Theorem \ref{th:Modular1}) and hence the three points $a,b,c$ generate a subgeometry that is a
homomorphic image
of the projective line $\Z \PP^1$ with its canonical base triple $(o,1,\infty)$ (Theorem \ref{th:Modular2}). 
Pairwise transversal triples $(a,b,c)$ in a Jordan geometry correspond to {\em unital Jordan algebras};
classification of such triples in a given geometry is related to the classical {\em Maslov index} -- 
the Jordan algebras associated to a geometry are {\em isotopic} to each other, but in general not isomorphic.

\ssk
If the geometry is not self-dual, then  there are no closed transversal triples; a substitute is given by
{\em idempotent quadruples} which are defined by relations obtained by ``dissociating'' the projective
line $\Z\PP^1$ into two copies, and leading to a homomorphism $\Gl(2,\Z) \to \Aut(\cX)$ 
(Theorem  \ref{th:Idempot}).
From an algebraic point of view, this corresponds to a geometric version of the {\em Peirce decomposition
with respect to an idempotent in a Jordan pair}, cf.\ \cite{Lo75}.

\subsection{Tangent objects}\label{subsec:Diff}
In the second half of this work,  we investigate the relation between Jordan structure maps $J$ and
{\em tangent objects} (Jordan algebras, pairs and triple systems).
We have divided the text into two main parts,  in order to highlight  the four  layers of
the axiomatic structure:

\ssk
(1)  a rather weak (non-)incidence structure, called  {\em transversality},

(2)  the general datum of one or several {\em structure maps},

(3)  certain {\em identities} (associative, Jordan,...) satisfied by the structure maps,

(4)  a {\em regularity hypothesis}, allowing to define ``tangent maps''  of structure maps.

\ssk
\nin
While the structures on levels  (1), (2) and (3) only use universal algebra and the language of classical geometry of  point sets, 
on level (4) one has to make a methodological choice:
either regularity is formalized  by some sort of  {\em differential calculus}, such as in classical Lie theory, or
it may be achieved by {\em extending the domain of definition of structure maps to non-transversal 
tuples}, as done in \cite{BeKi1} for the $M$-map.
However, at present we do not know how to apply this second method to the $J$-maps; so we have to follow the
first method, leaving the link with the second method as an open (and very important) topic for future research. 

\ssk
In order to use ``algebraic differential calculus'' in full generality, applying also to geometries defined over   $\Z$ as in the
present setting, we have introduced in \cite{Be14} the concept of {\em Weil spaces and Weil manifolds}.
These can be seen as a much more conceptual version of the algebraic differential calculus  already used in
\cite{Be02}, also  generalizing the so-called {\em Weil functors} defined for usual manifolds in \cite{KMS}.
We refer to the introduction of \cite{Be14} and to Subsection \ref{ssec:Weils}   of this work for more details;
here, let us just say that, using this calculus,  the ideas already present in \cite{Be00, Be02} translate fairly
directly into an algebraic language, allowing to attach to a geometry with base point a
{\em
$3$-graded Lie algebra $\g=\g_1 \oplus \g_0 \oplus \g_{-1}$}.
Now, it is well-known that such a Lie
algebra corresponds to a {\em linear Jordan pair} $(V^+,V^-)=(\g_1,\g_{-1})$, which is the tangent object
saught for -- at least, if the base ring $\K$ has no $2$- and $3$-torsion. 
In the remaining case,  we need to
work with {\em quadratic Jordan pairs} as defined in \cite{Lo75} --  we define quadratic maps
$Q^\pm$ that contain more information than the trilinear bracket derived from the Lie algebra;
the proof that these maps satisfy the Jordan identities (JP1) -- (JP3) from \cite{Lo75} 
follows the lines of work by O.\ Loos (\cite{Lo79}).

\subsection{Back and forth}
We can {\em reconstruct} a Jordan geometry from  its  Jordan pair --
in case $2$ is invertible in $\K$, this follows  from the corresponding result in \cite{Be02},
using {\em midpoints in affine spaces} (Theorem \ref{th:gpg}); if $2$ is not invertible in $\K$, 
the construction is similar, but more involved  (Theorem \ref{th:existence}).
Summing up, just as in classical Lie theory, we can go back and forth from Jordan geometries to Jordan pairs
and -algebras. 
In Section \ref{sec:Fomulae} we give ``explicit formulae'' for the maps $J^{xz}_a$, generalizing the
formulas in terms of homographies  given at the beginning of 
this introduction for the projective line, but we leave a more systematic 
study of  this correspondence for later work.
%; here, we content 
%ourselves by giving a collection of Jordan theoretic formulas describring the global structure map $J$ in terms of the 
%data of the Jordan pair (section \ref{sec:Fomulae}).

\begin{acknowledgment}
I thank the unknown referee for helpful comments and remarks.
\end{acknowledgment}

\begin{notation} We use the following typographic conventions in mathematical formulas:

\begin{itemize}
\item
calligraphic letters denote  ``geometric point spaces''  $\cX,\cY,\cD,\cU_a,\ldots$, and $\cU_{a,b} := \cU_a \cap \cU_b, \ldots$
\item
boldface letters denote transformation groups $\bfG,\bfU,\bfP,\Aut(\cX),...$,  and stabilizers
are denoted by  $\bfG_x$, $\bfG_{x,y} =\bfG_x \cap \bfG_y$,
\item
small italics denote  elements $x \in \cX$, $a \in \cY$, $g \in \bfG, \ldots$
\item
capital italics denote  structure maps $J,M,S$, 
 $B$ (Bergman operator), but also:
 $D$ differential, $T$ tangent, $V= (V^+,V^-)$ pair of $\K$-modules,
\item
blackboard letters:  $\K$ is a fixed base ring (think of $\K=\R$ or $\K=\Z$), and are
 $\bA,\bB,\ldots$ scalar extensions of $\K$  ($\K$-Weil algebras); 
 $T\K = \K[X]/(X^2)$ is the tangent ring of $\K$,
\item
underlined symbols  $\ul \cX, \ul \bfG,\ul J,\ldots$ are functors from $\K$-Weil algebras to the respective set-theoretic object, and
the corresponding scalar extended set theoretic object is  denoted by  $\cX^\bA$, $\bfG^\bA,J^\bA,\ldots$;
tangent bundles are then $T\cX = \cX^{T\K}$, $T\bfG = \bfG^{T\K}$,
$T(\bfG/\bfP) = (T\bfG)/(T\bfP)$, etc. 
\end{itemize}
%Terminology: actions: left right middle ? inversive ? symmetry ?
\end{notation}

\setcounter{tocdepth}{1}
\tableofcontents

\section*{FIRST PART:
GEOMETRIES WITH INVERSIVE ACTIONS}

\msk

\section{Transversality relations, splittings, dissociations}

\subsection{Transversality relations}
A {\em transversality relation} on a set $\cX$ is a binary relation on $\cX$, 
that is, a subset $\cD_2 \subset (\cX \times \cX)$; we write also
$x \top a$ if $(x,a) \in \cD_2$, and this relation is assumed to be

\ssk
-- {\em symmetric} : $x \top a$ iff $a \top x$,

-- {\em irreflexive}:  $x$ is never transversal to itself.  
% terminology cf http://en.wikipedia.org/wiki/Binary_relation
% need not require that $x^\top$ is non-empty: such $x$ are never reached by chains...singletons
% living in the exterior of our world

\ssk \nin
For the sets of elements transversal to one, resp.\ to  two given elements, we write
\begin{equation}
\cU_x:=x^\top := \{ a \in \cX \mid \, a \top x \}, \qquad
\cU_{ab}:= \cU_a \cap \cU_b .
\end{equation}
The relation $\top$ is called {\em non-degenerate} if $\cU_x = \cU_y$ implies $x=y$.

\ssk
{\em Homomorphisms} of sets with transversality relation are maps
$f:\cX \to \cY$ preserving transversality:
$x \top a$ implies $f(x) \top f(y)$.

\subsection{Grassmannians}\label{ssec:Grass-0}
The standard example of transversality is given by the {\em Grassmannian} $\Gras(W)$ 
of all submodules of some right $\bA$-module $W$
 with the relation:
$x \top a$ iff $V=x \oplus a$ (here $\bA$ may be a possibly non-commutative ring).
The {\em Grassmannian of type $E$ and co-type $F$} is the space
\begin{equation}
\Gras_E^F (W) := \bigl\{ x \in \Gras(W) \mid \, x \cong E, \, W/x \cong F \bigr\}
\end{equation}
of submodules isomorphic to $E$ and such that $W/x$ is isomorphic to $F$ (as modules),
where $W = E \oplus F$ is some fixed decomposition. 
Then the space
\begin{equation}\label{eqn:GrGeo}
\cX = \Gras_E^F(W) \cup \Gras_F^E(W)
\end{equation}
inherits a non-trivial transversality relation: for any $x \in \cX$, we have $x^\top \subset \cX$.
In particular, we get the {\em projective geometries}
\begin{equation}\label{eqn:ProjGeo}
\bA \PP^n \cup (\bA \PP^n)' := 
\Gras_\bA^{\bA^n}(\bA^{n+1}) \cup
\Gras_{\bA^n}^\bA(\bA^{n+1}) .
\end{equation}
If $\bA$ is a field or skew-field, this is a ``usual'' projective space together with its dual
space of hyperplanes (and ``transversal'' means the same as ``non-incident'', and the relation $\top$
is non-degenerate); 
however, if $\bA$ is a ring, such as $\bA=\Z$, then these geometries show some rather
unusual features (cf.\ the article by Veldkamp on Ring Geometries in \cite{Bue}).

\subsection{Transversal chains and connectedness}
Let $n\in \N$, $n>1$.
A {\em transversal chain of length $n$ in $ \cX$} is a sequence
$(x_1,\ldots,x_n)$ of elements of $\cX$ such that
$x_{i+1} \top x_i$ for $i=1,\ldots,n-1$ (equivalently, $x_i \in U_{x_{i-1},x_{i+1}}$). 
A transversal chain is called {\em closed} 
%{\em necklace} is a chain such that $x_n \top x_1$. 
%see for a different use  http://en.wikipedia.org/wiki/Necklace_(combinatorics)
if $x_n \top x_1$. 
We denote by
\begin{align}
\cD_n &=\{ (x_1,\ldots,x_n) \in \cX^n \mid \, \forall i=1,\ldots,n-1: \, x_{i+1} \top x_i \},
\cr
\cD_n' &= \{ (x_1,\ldots,x_n) \in \cD_n  \mid \, x_n \top x_1 \}  
\end{align}
the set of  transversal chains, resp.\ of closed transversal chains, of length $n$ in $\cX$. 
A chain of length two is also called a {\em transversal pair},  a chain of length three is
a {\em transversal triple}, and a closed transversal chain
 of length three is a {\em pairwise transversal triple}. 
A {\em chain joining two elements $x,y \in \cX$} is a finite chain
$(x_1,\ldots,x_n) \in \cD_n$  such that $x_1=x$, $x_n=y$.
We say that $\cX$ is {\em connected} if, for each $x,y \in \cX$, there is a chain
joining $x$ and $y$.
We may also define {\em connected components}:
the relation defined by ``$x \sim y$ iff there is a chain joining $x$ and $y$'' is an equivalence
relation; its equivalence classes are the {\em connected components of $\cX$}.

\ssk
For instance, Grassmannians $\cX = \Gras(V)$ are in general not connected; if $\K$ is a field
and $V = \K^n$, then its connected components are of the form (\ref{eqn:GrGeo}).
%, and each such  component is connected of stable rank one. 

\subsection{Duality: splitting, and antiautomorphisms}

Assume $\top$ is a transversality relation on $\cX$.
 A {\em splitting  of $\cX$} 
is a decomposition into  a disjoint union $\cX = \cX^+ \dot\cup \cX^-$ such that 
for all $a \in \cX^-$, we have $a^\top \subset \cX^+$, and for all $x \in \cX^+$,
we have $x^\top \subset \cX^-$.
Equivalently,  chains with odd length end up in the same part ($\cX^+$ or $\cX^-$) 
they started in, and chains with even length end up in the other.
We then say that $\cX^+$ and $\cX^-$ are {\em dual to each other}.

\ssk
We say that $(\cX^+,\cX^-)$ is {\em connected of stable rank one} if for each $(x,y) \in (\cX^\pm)^2$ 
there is $a\in \cX^\mp$ such that $x,y \in U_a$; equivalently, $a \in U_{xy}$, so $U_{xy}$ is not
empty. 

\ssk
Spaces with splitting $(\cX^+,\cX^-)$ form a category:
morphisms $g$ preserve transversality and the given splitting
(that is, $g(\cX^\pm)\subset \cY^\pm$),  so we have well-defined restrictions
$$
g^\pm:\cX^\pm \to \cY^\pm . 
$$
In presence of a splitting, we may also define {\em anti-homomorphisms}:
these are pairs of maps exchanging the components,
$\cX^\pm \to \cY^\mp$, i.e.,  morphisms from $(\cX^+,\cX^-)$ to the {\em opposite splitting}
of $\cY$.

\subsection{Self-dual geometries and closed transversal triples}
We say that a connected geometry $(\cX,\top)$ is {\em self-dual} if it 
does not admit any non-trivial splitting.
This is the case if in $\cX$ there is a closed transversal chain of odd length (at least three);
the converse is true as well. 
We say that $\cX$ is {\em strongly self-dual} if there is a closed chain of length three, that is,
a pairwise transversal triple $(a,b,c)$.

\ssk
In the 
example of a Grassmannian (\ref{eqn:GrGeo}), the indicated decomposition is a splitting
if $E$ and $F$ are not isomorphic as modules. Typical 
antiautomorphisms are then given by orthocomplementation maps. 
On the other hand, if $E \cong F$ as modules,  then 
there exists
 a pairwise transversal  triple $(E,F,D)$ where $D$ is the diagonal of $E \oplus F$, after some
 fixed identification of $E$ and $F$, 
 and hence $\Gras_E(E \oplus E)$ does not admit any splitting
 (this is the case, in particular, for  the {\em projective line}
 $\bA \PP^1$).
 In this case one may introduce an ``artificial splitting'', as follows.

%rk: \url{http://en.wikipedia.org/wiki/Dissociation_(psychology)}
%see also:  Split personality (disambiguation).

\subsection{Duality: dissociation}\label{ssec:Diss}
A {\em dissociation} of a space $(\cY,\top)$ is 
the disjoint union $\cX$ of two  copies $\cX^+$ and $\cX^-$ of $\cY$, where we
 define a transversality relation on $\cX$ by declaring, for
$x \in \cX^\pm$, the set $x^\top$ to be the set of elements $a$ in
$\cX^\mp$ such that $a$ and $x$ are transversal in $\cY$.
Obviously, this defines a transversality relation on $\cX$, and
$\cX = \cX^+ \cup \cX^-$ is a splitting.

%[add comment on dissociation of projective line ?]
\begin{comment}%%%%%%%%%%%%%%%%%%%%%%%%%%%%%%%%%%%%%%
In practice, the {\em dissociation of a projective line} $\bA \PP^1$ is most important
(for a commutative ring, such as $\bA=\Z$):
indeed, an abstract projective line $\bA \PP^1$, and a ``concrete'' projective line
$\ell \subset \bA \PP^2$ behave differently. Namely, if you choose also a projective line
$\ell'$ in the dual projective plane $(\bA \PP^2)' = \Gras_2^1(\bA^3)$, such that $\ell$ and
$\ell'$ correspond to a direct sum decomposition of $\bA^3$, then the pair of lines
$(\ell,\ell')$ represents the dissociation of an abstract projective line. 
Now, the group $\SL(2,\bA)$ acts on $\bA \PP^1$ in the usual way, and using the decomposition
$\bA^3 \cong \bA^2 \oplus \bA$ induced by $(\ell,\ell')$, every element of $\SL(2,\bA)$ acts also
on $\bA \PP^2$, by extending each element of $\SL(2,\bA)$ to a $3\times 3$-matrix in the canonical way.
But then the square of the matrix 
$\bigl(\begin{smallmatrix}
0&1\\ -1&0
\end{smallmatrix} \bigr)$
%$\begin{pmatrix}0&1\cr-1&0 \end{pmatrix}$
acts trivially on the abstract projective line, but not on the pair $(\ell,\ell')$.
Anticipating somewhat, we may say that the automorphism group of the dissociated line
is a central extension of the automorphism group of the ``usual'' projective line; this is
related to  the {\em Peirce decomposition with respect to an
idempotent} (see below, \ref ).
\end{comment}

\section{Jordan structure maps}\label{sec:JSM}

\subsection{Structure maps in general}
Assume $(\cX,\top)$ is a space with transversality relation, and let $n\in \N$.
An {\em $n+1$-ary structure map (with domain $\cD_n$)} 
is a map 
$$
S : \cD_n \to \End(\cX,\top)
$$
attaching to each chain $x=(x_1,\ldots,x_n)$ a map  $S(x):\cX \to \cX$ preserving
transversality.   
In the sequel we will mainly consider structure maps such that $S(x)$ is a bijection,
and the case of ternary and quaternary structure maps will be most important:
for a ternary structure map we use also the notation
$$
S^a_x := S(x,a),
$$
and for a quaternary structure map
$$
S^{xz}_a := S(x,a,z) .
$$
Sometimes we view $S$ as a map of $n+1$ arguments, defined by
$$
S(x_1,\ldots,x_n,x_{n+1}):= (S(x_1,\ldots,x_n))(x_{n+1} ) =
S^{x_1 x_3 \ldots}_{x_2 x_4\ldots} (x_{n+1}) 
$$
{\em Structure maps with domain $\cD_n '$} are defined similarly.

\ssk
{\em Morphisms of spaces with structure map} are maps preserving transversality
and commuting with structure maps in the obvious sense. 
The group of automorphisms of $(\cX,\top,S)$ is denoted by $\Aut(\cX)$,
$\Aut(\cX,S)$, or $\Aut(\cX,\top,S)$, 	according to the context. 
Other categorial notions can be defined for spaces with structure maps, such as
{\em subspaces, direct products...}

\subsubsection{Structure maps and duality}
If $\cX = \cX^+ \cup \cX^-$ is a splitting of $\cX$, 
then let $\cD_n^\pm$ be the set of chains of length $n$ starting in
$\cX^\pm$, so that $\cD_n = \cD_n^+ \cup \cD_n^-$.
Then, by restriction to $\cD^\pm_n$, a structure map $S$ gives rise
to two parts of $S^\pm$ of the structure map.
Thus one recovers the notation used in \cite{Be02}.

\subsection{Jordan structure map: definition, examples, and first properties}

\begin{definition}\label{def:Jordanstructure}
A {\em  Jordan structure map}  on a space $(\cX,\top)$ is  a quaternary
structure map 
$$
J:\cD_3  \to \End(\cX,\top), \qquad (x,a,z) \mapsto J^{xz}_a
$$
 such that the following {\em Jordan identities} hold:

\ssk

\begin{enumerate}
\item[(IN)] \emph{involutivity}: $J^{xz}_a \circ J^{xz}_a = \id_\cX$
\ssk
\item[(IP)]
\emph{idempotency}: $J^{ab}_c(c)=c$, $J^{ab}_c(a)=b$, $J^{ab}_c(b)=a$
\ssk
\item[(A)] \emph{associativity:} $J^{xz}_c J^{uv}_c J^{ab}_c = J_c^{J_c^{xa}(v), J_c^{bz}(u)}$
\ssk
\item[(D)] \emph{distributivity:}   $J^{xz}_c \circ J^{uv}_b \circ J^{xz}_c = 
J^{J_c^{xz}(u),J_c^{xz}(v)}_{J_c^{xz}(b)} $, 
that is, $J^{xz}_c \in \Aut(\cX,J)$,
\ssk
\item[(C)] \emph{commutativity:}  $J^{ab}_c=J^{ba}_c$
\ssk
\item[(S)] \emph{symmetry:}  $J^{xx}_a = J^{aa}_x$.
\end{enumerate}

\ssk
\nin
When $a \in \cX$ is considered as fixed, we use also the notation
\begin{equation} \label{eqn:(xyz)}
(xyz)_a:= J_a^{xz}(y) . 
\end{equation}
Using this, the last two properties from (IP) are written
$(aab)_c=a$, $(abb)_c=b$, explaining the terminology.
\end{definition}

\subsubsection{Example: projective quadrics}\label{sssec:Quadrics}
Assume $\cX = Q$ is a projective quadric in the projective space $\PP(W)$ of a vector space $W$.
Two elements of $Q$ are called {\em transversal} if the line joining them in $\PP(W)$ is a secant,
i.e., not a tangent line of $Q$. 
If $x=[v]$ and $z=[w] \in Q$ are transversal to $a=[u] \in Q$, then there exists a  unique orthogonal map
$I^{xz}_a:W \to W$ exchanging $[x]$ and $[z]$, fixing $u$, and acting as $-1$ on the orthogonal
complement of $\mbox{Span} (u,v,w)$.  
We define $J^{x,z}_{a}$ to be the restriction to $Q$ of the projective
map induced by $I^{x,z}_a$. The family of maps thus defined
satisfies the properties given above (and has some more specific properties which 
  we intend to investigate in more detail in subsequent work). 
  
\subsubsection{Example: Grassmannians}
In Section \ref{sec:ASM} we will define the Jordan structure map of a 
{\em Grassmannian geometry} and prove that it satisfies the Jordan identities.

\begin{lemma}[Torsor action]\label{la:ta}
Given a Jordan structure map, the set  $\cU_a$ with product given by (\ref{eqn:(xyz)})
 is a commutative torsor, and the map
$\cU_a \times \cU_a \to \Bij(\cX)$, $(x,z) \mapsto J^{xz}_a$ is an inversive  torsor action.
\end{lemma}

\begin{proof} 
If $x,y,z \top a$, then also $J^{xz}_a(y) \top J^{xz}_a(a)=a$, hence $\cU_a$ is stable under
$(xyz)_a$. Now
the idempotent identity of a torsor holds, as remarked in the definition, and  para-associativity follows from this and from (A) 
(lemma \ref{TorsorLemma}).  
The properties of an inversive  torsor action are precisely the axioms (A) and (C).
\end{proof}

\nin As a useful application of the lemma, by Appendix \ref{la:Transp}, we have the following
{\em transplantation formula} for the symmetries: for all $x,o,z \top a$,
\begin{equation}\label{eqn:Transp}
J_a^{xz} =
J_a^{xo} J_a^{oo} J_a^{zo} = J_a^{J_a^{xz}(o),o} .
\end{equation}

\begin{lemma}\label{la:morph}
A map $g:\cX \to \cX$ is an endomorphism of $(\top, J)$ iff, for all $a \in \cX$,
the restriction $g\vert_{\cU_a}:\cU_a \to \cU_{g(a)}$ is well-defined and is a 
torsor-homomorphism.
\end{lemma}

\begin{proof}  This is a re-writing  of 
$x \top a \Rightarrow g(x) \top g(a)$ and of
$g (J^{xz}_a ) (y)= J^{gx,gz}_{ga} (gy)$.
\end{proof}

\begin{lemma}[Symmetric space action]\label{la:ssa}
 The set $\cU:=\cU_{ab}$ is
 stable under the map
$$
\mu:=\mu_{ab}: \cU \times \cU\to \cU, \quad
(x,y) \mapsto \mu(x,y):= s_x(y):= J^{ab}_x(y) 
$$
which turns it into a reflection space, called
the {\em reflection  space associated to $(a,b)$},
%(R1) $\mu(x,x)=x$,
%(M2) $\mu(x,\mu(x,y))=y$,
%(R2) $s_x s_z s_x = s_{s_x (z)}$,
and this reflection space has a symmetry action on $\cX$ given by
$S_x:=J^{ab}_x$.
\end{lemma}

\begin{proof}
By the preceding lemma,  and since $J^{ab}_x$ exchanges $a$ and $b$, the restricted maps
$J^{ab}_x: \cU_a \to \cU_b$, 
$J^{ab}_x: \cU_b \to \cU_a$,
are well-defined torsor morphisms, inverse to each other.  Thus $\mu_{ab}$ is well-defined.
Properties (R1) and (R2) of definition \ref{ReflectionDef} are immediate. To prove  (R3), as in the preceding proof 
it is seen that $J_x^{ab}$ is
an automorphism of $\mu$. Thus $\cU_{ab}$ is a reflection  space, and it acts on $\cX$ by a symmetry action
by axiom (D), read  with  $(u,v)=(x,z)$.
\end{proof}

\begin{lemma}[Compatibilty]
The reflection   space $\cU_{aa}$ is the same as the
abelian group $\cU_a$ with its usual  inversion maps.
\end{lemma}

\begin{proof} This follows directly from the symmetry property (S):
in $\cU_{aa}$ the symmetric element
of $y$ with respect to $x$ is
$s_x(y) = J^{aa}_x(y)$, and in $\cU_a$ it is
$(xyx)_a = J^{xx}_a(y)$.
\end{proof}

\begin{theorem}[The polarized reflection  space]\label{th:PSS}
The set $\cD_2$ of transversal pairs becomes a reflection space with the law
\[
s_{(x,a)} (y,b) := \bigl( J^{xx}_a (y), J^{aa}_x(b) \bigr).
\]
The same formula defines a symmetry action of the reflection  space $\cD_2$ on $\cX^2$.
The exchange map $\tau:\cX^2 \to \cX^2$, $(x,a)\mapsto (a,x)$ is an automorphism of the reflection space $\cD^2$ and
of the action.
\end{theorem}

\begin{proof}
Everything follows easily from the axioms (In), (IP, (D).
\end{proof}

\nin
The reflection space $\cD_2$ contains flat subspaces (as defined below, \ref{ssec:flat})
for $a=b$ fixed (or $x=y$ fixed), but  is not
flat itself. It corresponds to the {\em twisted polarized symmetric spaces} from \cite{Be00}.

 \subsection{Some categorial notions}\label{ssec:cat}
Most categorial notions are defined in an obvious way -- cf.\
\cite{Be02, BeL08},  and we refer to loc.\ cit.\ for more details:

\subsubsection{Morphisms}
They can be characterized as ``locally $\Z$-affine maps'' (Lemma \ref{la:morph}),
or, similarly, as morphisms of the family of  ``local''  reflection spaces.

\subsubsection{Inner automorphisms, groups of automorphisms}
The subgroup
\begin{equation}
\bfG := \bfG(\cX):= \langle J^{xz}_a \mid (x,a,z) \in \cD_3 \rangle \subset \Aut(\cX,J)
\end{equation}
generated by all inversions will be called the {\em group of inner automorphisms}. Stabilizers of one element $x \in \cX$,
resp.\ of a transversal pair  $(x,a)\in \cD_2$, are written
\begin{equation}
\bfG_x := \{ g \in \bfG \mid g(x)=x \}, \qquad   \bfG_{x,a}:= \bfG_x \cap \bfG_a .
\end{equation}
Note that $\bfG_a$ acts $\Z$-affinely on $\cU_a$, and
$\bfG_{x,a}$ acts $\Z$-linearly on $(\cU_a,x)$ and  $(\cU_x,a)$.
 
\subsubsection{Base points:}
A {\em base point} is a fixed  transversal pair $(x,a)$; we then often write $(o,o')$ or $(o^+,o^ -)$.
For the stabilizer groups we sometimes write also
\begin{equation}\label{eqn:stabilizers}
\bfP:= \bfG_{o'}, \qquad \bfH := \bfG_{o,o'} .
\end{equation}

\subsubsection{Duality:} In presence of a splitting, we define structure maps $J^\pm$,
see above.

\subsubsection{Direct products:} 
direct product of transversality and of structure maps

\subsubsection{Subspaces:} 
subsets stable under structure maps

\subsubsection{Intrinsic subspaces (inner ideals):} 
subsets $\cY \subset \cX$ such that $J^{xz}_a(y) \in \cY$
whenever $x,y,z \in \cY$ and $a \in \cX$;
this can be interpreted in two ways:
$U_a \cap \cY$ is an affine subspace, for all $a \in \cX$ (point of view taken in \cite{BeL08}), or:
$\cY$ is an invariant subspace of the symmetry action of $U_{xz}$,
for all $x,z \in \cY$.

\subsubsection{Flat geometries:}  \label{ssec:flat}
given by two abelian groups $(V_1,+),(V_{-1},+)$, $\cX = V_1 \cup  V_{-1}$ (disjoint union), 
$a \top x$ iff $a \in V_{\pm 1}$ and $x \in V_{\mp 1}$, 
$J^{xz}_a (y)=x-y+z$, $J^{xz}_a(b)= 2a-b$ for $x,y,z \in V_{\pm 1}$, $a,b \in V_{\mp 1}$.

\subsubsection{Congruences and quotient spaces:} defined as in \cite{Be02}, following
\cite{Lo69}, III.2.

% define suitable a center congruence

\subsubsection{Polarities}
A {\em polarity} is an automorphism $p \in \Aut(\cX)$ which is of order two: $p^2 = \id_\cX$, and
has  {\em non-isotropic elements}: there is $x$ such that $p(x) \top x$.
In other words, $(x,p(x)) \in \cD_2$, so the graph of $p$ has non-empty intersection with $\cD_2$.

\begin{theorem}
Assume $p$ is a polarity of $(\cX,\top,J)$.
 Then the set 
$$
\cX^{(p)} = \{ x \in \cX \mid \, p(x) \top x \},
$$
is stable under the law $(x,y) \mapsto J^{xx}_{p(x)}(y)$, which turns it into a reflection space.
\end{theorem}

\begin{proof}
Indeed, this space can be realized as sub-reflection  space of the polarized space $\cD_2$
(Theorem \ref{th:PSS}) fixed under the involution $p\tau = \tau p$, by the imbedding
$x \mapsto (x,p(x))$. (This is the analog of \cite{Be02}, Theorem 4.2).
\end{proof}
%
\begin{comment}%%%%%%%%%%%%%%%%%%%%%%%%%%%%%%%%%%%%%%%%%
The symmetric spaces $\cX^{(p)}$ are in general not homogeneous under $\Aut(\cX)$:
the {\em symmetric spaces with base point}
$(\cX^{(p)},x)$, $(\cX^{(p)},y)$ are called {\em isotopic};
in general they are not isomorphic to each other. Note that the {\em space of polarities},
\begin{equation}
{\rm Pol}(\cX) := \{ p \in \Aut(\cX) \mid \, p^2 = \id_\cX, \, \exists x \in \cX : (x,p(x)) \in \cD_2 \} \, ,
\end{equation}
is itself a reflection space: it is stable under  $(p,q) \mapsto pqp$.
If $(a,b,c)$ is a pairwise transversal triple, then $p=J^{ab}_c$ is a polarity, called an
{\em inner polarity} (cf.\ \cite{Be03}).
\end{comment}
% a suivre! this should be the "bounded realization" of $U_{ab}$ ! 
% for all $x$ in the sym spaces, $J_{J_c^{ab}x}_{xx}$ is conjugate, bot not equal, to $J^{ab}_c$ !

 \section{Translations}\label{sec:Tran}
 
To each of the torsors $\cU_a$ corresponds a translation group $\bfT_a$, acting on $\cX$ by inner automorphisms.
The rich supply of inner automorphisms permits to prove transitivity results, and to define special elements
of stabilizer groups (Bergman operators). Together, this gives a good knowledge of the ``canonical atlas of $\cX$''.
 
\begin{definition}%[Translations] 
Fix $a \in \cX$. According to Lemma \ref{L-Lemma},
for all $x,z \in \cU_a$,  the map
\begin{equation}
L^{xz}_a:=J_a^{xu} J_a^{uz} = J_a^{ux} J_a^{zu} 
\end{equation}
 called {\em (left) $a$-translation}, does  not depend on the choice
of $u \in \cU_a$ (in particular, we may choose $u=x$ or $u=z$).  The $a$-translations form
a commutative group,
\begin{equation}
\bfT_a := \{ L^{xz}_a \mid \, x,z \in \cU_a  \} \subset \bfG_a ,
\end{equation}
called the {\em $a$-translation group}.  It acts on $\cX$ by its natural left action. 
\end{definition}

\nin
Note that $\bfT_a$ is isomorphic to $(\cU_a,o)$, for any origin $o \in U_a$, and,
if $x,y,z \in \cU_a$, then we have usual properties, such as
``Chasles relation'', and the link with the symmetries:
\begin{equation}
L_a^{xy} L_a^{yz} = L_a^{xz}, \quad (L_a^{xy})\inv = L_a^{yx},\quad 
L_a^{xy}(z) = (xyz)_a = J_a^{xz}(y) .
\end{equation}

\begin{lemma}\label{la:normal}
The translation group $\bfT_a$ is a normal subgroup of $\bfG_a$, and, for any $x \in \cU_a$, it is the kernel of the group
homomorphism
$$
D:=D^{x,a}: \bfG_a \to \bfG_{a,x}, \quad g \mapsto D(g):= L_a^{x,gx} \circ g .
$$
This homomorphism has a section given by the natural inclusion, and hence we have an exact sequence of groups
\begin{equation}\label{eqn:sequence}
0 \to \bfT_a \to \bfG_a \to \bfG_{x,a} \to 0 .
\end{equation}
\end{lemma}

\begin{proof} For any $g\in \bfG$,
 $g \circ L_a^{xz} \circ g\inv = L_{ga}^{gx,gz}$, which implies that $\bfT_a$ is normal in $\bfG_a$.
Clearly, $D(g)a=a$ and $D(g)x=x$, so the map $D$ is well defined, and it is a morphism:
$D(g)D(h)=L_a^{x,gx} g L_a^{x,hx} h = L_a^{x,gx}L_a^{gx,ghx} gh= D(gh)$, and its kernel is the set of $g$
such that $g= L_a^{gx,x}$, that is, the translation group. 
Obviously, the inclusion is a section of $D$.
\end{proof}

\begin{lemma}[Triple decomposition]\label{la:tripledecomposition}
Fix a base point $(x,a )= (o,o') \in \cD_2$ and let $\bfT :=\bfT_{o'}$, $\bfT':=\bfT_o$. 
Then each element of  the {\em big cell of $\bfG$}
\begin{equation}
\Omega :=\Omega^{o,o'} := \{ g \in \bfG \mid g(o) \top o' \}
\end{equation}
admits a unique   {\em triple decomposition} into a translation, a $\Z$-linear part, and a ``quasi-translation'',
that is,  $\Omega = \bfT \cdot \bfG_{o,o'} \cdot  \bfT'$:
\begin{equation}\label{eqn:triple}
 g = L_{o'}^{t,o}  h L_o^{t',o'} , \qquad \mbox{ with  } t \in \bfT , h \in \bfG_{o,o'}, t' \in \bfT'  .
 \end{equation}
\end{lemma}

\begin{proof}
The  decomposition is unique:
 necessarily, $t=g(o)$, $t'=-g\inv(o')$ and hence 
\begin{equation}\label{Den-eqn}
h= D(g) :=D^{o,o'}(g) := L_{o'}^{o,g(o)} \circ g \circ L_o^{g\inv (o'),o'}  .
\end{equation}
In order to prove existence, it suffices to check that $D(g)$ stabilizes $o$ and $o'$,
and this is done as in the preceding proof.  
\end{proof}

\begin{definition}
In the situation of the lemma, we say that $\bfT$ acts by translations on $V:=\cU_{o'}$, and $\bfT'$
{\em acts by quasi-translations}. We use also
 the notation\footnote{In Jordan theory, $x^a$ is called the {\em quasi-inverse}, but we prefer to use 
 this  term  for $J^{ao'}_o(x)$, which describes an map
of order two, and thus corresponds much better to some kind of inverse.} $x^a :=  L_o^{ao'}(x)$,
and we say that the pair $(x,a) \in \cU_{o'}\times \cU_o$ is {\em quasi-invertible}
if $x^a \in \cU_{o'}$.
\end{definition}

\nin
The preceding definition corresponds to the choice of considering $\cU_{o'}$ as ``space'' and
$\cU_o$ as ``dual space''. But of course, things can be turned over: the conditions
$g.o\top o'$ and $g(o) \in V$ are equivalent, and $\Omega^{o',o}=(\Omega^{o,o'})\inv$.
The element $D(g)$ defined by (\ref{Den-eqn})
will be called  {\em denominator of $g$ with respect to $(o,o')$}. 
Note that its definition  is compatible 
with the one from Lemma \ref{la:normal}.

\begin{lemma}
The denominators satisfy the
{\em cocyle relation}
\begin{equation}\label{eqn:Coc}
D(gh) = D(g L_{o'}^{ho,o}) D(h) . 
\end{equation}
\end{lemma}

\begin{proof}
If $h=sD(h) s'$ and $gs = t D(gs) t'$, where $s = L_{o'}^{ho,o}$, then
\[
gh = g sD(h)  s' = t D(gs)  D(h) s' 
\]
whence, by uniqueness of the decomposition, $D(gh)= D(gs) D(h)$.
\end{proof}

The projective group $\PP\Gl(p+q,\K)$ acts transitively on $\Gras_p(\K^{p+q})$, but not
on $\Gras_p(\K^{p+q})\cup \Gras_q(\K^{p+q})$ (since it preserves dimension) -- unless $p=q$,
in which case the geometry is self-dual. 
The following result generalizes these observations:

\begin{theorem}[Transitivity] \label{th:trans}
Assume $(\cX,\top,J)$ is connected, and fix a base point $(o,o')$. 
Let $X:=\cX$ and $M:=\cD_2$ if  $\cX$ is self-dual, and
$X:= \cX^+$ and $M:=\cD_2^+$ if $\cX $ admits a splitting
$(\cX^+,\cX^-)$. 
Then the action of $\bfG$ on $M$ and on $X$ is transitive:
$$
M   = \bfG/ \bfG_{o,o'} = \bfG / \bfH,\qquad  X  = \bfG/\bfG_o = \bfG/\bfP \, ,
$$
and every element $g \in \bfG$ has a (in general, not unique) 
decomposition
$$
g = t_1 s_1 \cdots t_n s_n \, h ,
% L_{o'}^{x_1 o} L_o^{a_1 o'} \ldots L_{o'}^{x_n o} L_o^{a_n o'} h
$$
with $t_i \in \bfT = \bfT_{o'}$ and $s_i \in \bfT' = \bfT_o$, $i=1,\ldots,n$, 
%$x_1,\ldots,x_n \in V = \cU_{o'}$, $a_1,\ldots,a_n \in V' = \cU_o$, 
and $h \in \bfG_{o,o'}$.
\end{theorem}

\begin{proof}
Both claims are proved by induction on length of chains joining two points.
Assume that  $(x,a,y,b)$ is a chain, so $x \top a$, $a \top y$, $y \top b$.
Then the element
\begin{equation}\label{eqn:La}
\La := \La_{yx}^{ba}:=
 L^{ba}_y \circ L^{yx}_a 
\end{equation}
has the properties
$\La(a)= L^{ba}_y (a)=b$ and
$\La(x)=L^{ba}_y (y) = y$, and hence maps $(a,x)$ to $(b,y)$.
Now the transitivity result follows by induction on the length of chains. 
Note, moreover, that $\La$ may be rewritten in the form
\begin{equation}\label{eqn:La'}
\Lambda= L^{ba}_y \circ L^{yx}_a = L_y^{ba} \circ L^{L_y^{ab} y,x}_x  =
L_y^{ba} \circ L^{L_y^{ab} y,x}_x  \circ L_y^{ab} \circ  L_y^{ba} =
L_b^{y,L_y^{ba}(x)} \circ L_y^{ba} ,
\end{equation}
thus (if $(y,b)=(o,o')$ is the base point) expressing $\La$ by an element of the desired form.
Again, the general decomposition now follows by induction.
\end{proof}

It follows that, if $\cX$ is conntected, all involutions $J^{xx}_a$ are conjugate to each other under
$\Aut(\cX)$. See Remark \ref{rk:midpoints}
 for a sufficient condition that ensures that also all $J^{xz}_a$ are conjugate
to each other.

\begin{remark}
 If $(\cX^+,\cX^-)$ is connected  of
stable rank one, the decomposition from the theorem exists
with $n=2$ and $a_2= o'$ (so we can write
 $G = \bfT \, \bfT' \, \bfT \, \bfH$; in the case of Hermitian symmetric spaces this is called ``Harish-Chandra decomposition''). 
\end{remark}

\begin{remark}
In the preceding statements and proofs, we could have replaced the letter ``$L$'' by ``$J$''; 
for instance, we have
\begin{align}\label{eqn:La''}
\Lambda_{xy}^{ab} =
L^{ab}_x L^{xy}_b & = J^{ab}_x J^{bb}_x J^{xx}_b J^{xy}_b = J^{ab}_x J^{xy}_b \cr
&=
J_x^{ab} J^{J_x^{ab}(x),y}_b =  J^{x,J^{ab}_x(y)}_a J_x^{ab} \cr
& = J_x^{a J_b^{xy}(b)} J_b^{xy} = J_b^{xy} J_y^{J^{xy}_b(a),b} \, .
\end{align}
\end{remark}

\begin{definition}[Bergman operator]\label{ssec:Berg}
A {\em quasi-invertible quadruple} is a closed chain of length four,
$(a,x,b,y) \in \cD_4'$. By closedness, we can define the element
$\La^{yx}_{ba} =L_b^{yx} L_x^{ba}$, having the same effect on $(a,x)$ as the element $\La_{yx}^{ba}$ 
from the preceding proof. It follows that $(\La^{yx}_{ba})\inv \La_{yx}^{ba}$ stabilizes $(a,x)$, i.e., belongs to $\bfG_{x,a}$.
This leads us to define,
for $(a,x,b,y) \in \cD_4'$, the {\em Bergman operator}
\begin{align} \label{eqn:B}
B^{xa}_{yb} 
 &: =(  \La^{yx}_{ba})\inv \La_{yx}^{ba} 
 = L^{ab}_x L^{xy}_b L^{ba}_y L^{yx}_a = \La_{xy}^{ab} \La_{yx}^{ba} \in H_{xa} \, .
\end{align}
\end{definition}

\nin  According to (\ref{eqn:La''}),  we have also the expression
\begin{equation}
B^{xa}_{yb} =  J^{ab}_x J^{xy}_b J^{ba}_y J^{yx}_a  \, .
\end{equation}
Note that 
\begin{equation}
(B_{xa}^{yb})\inv = B_{ax}^{by} .
\end{equation}
Obviously, the fourfold map $B$ is invariant under automorphisms $g \in \Aut(\cX,J)$.
If $(x,a)=(o,o')$ is chosen as base point, we also use the notation from Jordan theory
\begin{equation}
\beta(y,b) := B_{yb}^{o,o'} .
\end{equation} 
%In finite dimension over a field, the determinant of the linear operator $B(x,a,y,b)$ is a
%scalar invariant; on a projective line, it is the square of the cross-ratio.

\begin{lemma}\label{la:Berg}
Fix $(x,a)=:(o,o')$ as base point. Then $\beta(y,b)$ is a denominator, namely
$\beta(y,b) = D (L_o^{o'b} \, L_{o'}^{yo} )$.
In other terms, we have the relation
$$
L_o^{o'b} \circ L_{o'}^{yo} = L_{o'}^{L_o^{o'b}(y),o}  \circ \beta(y,b) \circ L_o^{o',L_{o'}^{oy}(b)} \, .
$$
In a similar way,
$$
J_o^{o'b} \circ J_{o'}^{yo} = J_{o'}^{J_o^{o'b}(y),o}  \circ \beta(y,-b) \circ J_o^{o',J_{o'}^{oy}(b)} \, .
$$
\end{lemma}

\begin{proof} As in (\ref{eqn:La''}),
$
\beta(y,b) =  L^{o'b}_{o} L^{oy}_b L^{bo'}_y L^{yo}_{o'} = 
L_{o'}^{o,L_o^{o'b}(y)} \, L_o^{o'b} \, L_{o'}^{yo} \, L_o^{L_{o'}^{oy}(b),o'}  = D(g)
$
for $g = L_o^{o'b} \, L_{o'}^{yo} $.
\end{proof}

%and: case of a projective line:
%$B(o,\infty,1,t)=B(1,t)= t-2t+t^2 = (1-t)^2$, so $B(x,a,y,b)$ is essentially the square of the
%usual cross-ratio 

%later $d_g(x) := (D_{g \circ L_{o'}^{x,o}})\inv$: quadratic polynomial 

\begin{definition}\label{def:atlas}
Let $X$ be a set and $V$ an abelian group ($\Z$-module). 
A {\em set-theoretic atlas of $X$ with model space $V$} is given by 
$\cA = (U_i,\phi_i,V_i)_{i\in I}$, where  for each $i$ belonging to an index set $I$,
$U_i \subset X$ and $V_i \subset V$ are non-empty subsets such that $X = \cup_{i\in I} U_i$,
and $\phi_i : U_i \to V_i$ is a bijection. 
The topology
generated by the sets $(U_i)_{i\in I}$ on $X$ is called the
{\em atlas-topology on $X$}. 
Given an atlas, we let for $(i,j) \in I^2$, 
$$
U_{ij}:= U_i \cap U_j \subset X , \qquad V_{ij}:= \phi_j(U_{ij}) \subset V ,
$$
and the {\em transition maps} belonging to the atlas are defined by
\begin{equation*}
 \phi_{ij}:= \phi_i\circ{\phi_j^{-1}} \vert_{ V_{ji} }:V_{ji}\to V_{ij} .
\end{equation*}
%They are bijections satisfying the {\em cocycle relations}
%\begin{equation}\label{eqn:cocycle}
 %\phi_{ii}=\id \quad  \mbox{and} \quad    \phi_{ij}\phi_{jk}= \phi_{ik} \ \mbox{ (where defined)}.
%\end{equation}
\end{definition}

\begin{lemma}\label{la:atlas}
Assume the Jordan geometry $\cX$ is connected, and define $X$ and $M$ as in Theorem \ref{th:trans}.
Fix a base point $(o,o') \in \cD_2$ and let $V:= \cU_{o'}$ and  $I:= \bfG$. Then
$\cA = (U_g,\phi_g,V_g)_{g\in \bfG} $
is an atlas on $X$ with model space $V$, where
$$
U_g = g(V), \quad V_g = V, \quad  \phi_g:U_g \to V, x \mapsto g\inv (x) .
$$
\end{lemma}

\begin{proof} Only the covering property is non-trivial, and this holds by Theorem \ref{th:trans}. 
\end{proof}

We call this atlas the {\em  canonical atlas of $X$}. 
It depends on the base point $(o,o')$; but, since the action of $\bfG$ on $M$ is transitive, this dependence is not
essential (it replaces the model space by an isomorphic one). 
The transition maps are given by
\begin{equation}
V_{g,h}:=  V \cap hg\inv V, \qquad \phi_{g,h} : V_{h,g} \to V_{g,h},\, v \mapsto g h\inv (v) .
\end{equation} 
If $\cX$ admits a splitting, then we define dually  an atlas of $\cX^-$ modelled on $V^-$.

\begin{comment}%%%%%%%%%%%%%%%%%%%%%%%%%%%%%%%%%%%%%%%%
\subsection{Transvections}\label{ssec:Transv}
For a closed quadruple $(x,a,z,b)$ we define  the  {\em transvection}
\begin{equation}\label{eqn:Transv}
Q_{xz}^{ab} := J^{ab}_x J^{ab}_z \, .
\end{equation}
This is the transvection $Q_{xz}=S_x S_z$ associated to the symmetry action of the symmetric space
$U_{ab}$ on $\cX$ (see Definition \ref{def:Transv}).
It maps $a$ to $a$ and $b$ to $b$, hence  acts affinely both on $U_a$ and
on $U_b$. 
These operators satisfy the fundamental formula (\ref{eqn:Fu}), and we
 have $(Q_{xz}^{ab})\inv= Q_{zx}^{ab} = Q_{zx}^{ba}$.
For $a=b$, $Q^{aa}_{xz}$ is a translation  for the torsor $U_a$:
\[
Q^{aa}_{xz}= J^{aa}_x J^{aa}_z = J^{xx}_a J^{xz}_a J^{xz}_a J^{zz}_a =
(L_a^{xz} )^2 .
\]
Fixing $(z,b)=(o,o')$ as base point,  one may thus
 think of $Q^{ab}_{xz}$ as a sort of ``deformation'' of the translation operator 
 $(L^{x,o}_{o'})^2(y)= 2x+y$.
 
 %, and
%\begin{equation}
%Q_{xz}^{ab} = L_x^{ab} L_b^{J_b^{xx}(y),y} L_y^{ba} .
%\end{equation}
\end{comment}%%%%%%%%%%%%%%%%%%%%%%%%%%%%%%%%%%%%%%%%%%%

\section{Associative structure maps}\label{sec:ASM}

As explained in the introduction, most examples of Jordan geometries are {\em special} in the sense that
they come from {\em associative geometries}. The following definition is a slight modification of the one
given in \cite{BeKi1}:

\begin{definition}\label{def:assgeo} 
An {\em  associative structure map} on a set $\cX$ with transversality relation $\top$ is  a quintary
structure map 
\[
M:\cD_4' \to \End(\cX),\qquad
(x,a,z,b) \mapsto M_{xz}^{ab}
\]
(we write also $(xyz)_{ab}:=M_{xz}^{ab}(y)$)
such that the following identities hold
% (writing also $(XYZ)_C = J^{XZ}_C(Y)$)

\ssk
\begin{enumerate}
\item[(1)] {\em symmetry}: $M_{xz}^{ab} = M_{ab}^{xz} = M_{ba}^{zx}$,
\ssk
\item[(2)]  {\em idempotency}: $M^{ab}_{xz}(x)=z$, $M^{ab}_{xz}(z)=x$, $M^{ab}_{xz}(b)=a$
 and $M^{ab}_{xz}(a)=b$,
\ssk
\item[(3)] {\em inverse:}  $M_{ab}^{xz}\circ M_{ab}^{zx}=\id_\cX$,
\ssk
\item[(4)]  {\em associativity:}  $M^{xz}_{ab} M^{uv}_{ab} M_{ab}^{rs} = M_{ab}^{(xvr)_{ab},(suz)_{ab}}$,
\ssk
\item[(5)] {\em distributivity:}  $M_{xz}^{ab} \circ M_{uv}^{cd} \circ (M_{xz}^{ab})\inv =
M_{(xuz)_{ab},(xvz)_{ab}}^{(xcz)_{ab},(xdz)_{ab}}$
(i.e., $M_{xz}^{ab}\in \Aut(M)$).
\end{enumerate}
\end{definition}

\nin
From idempotency and associativity, it follows that the set $\cU_{ab}$ is stable under the
ternary map $(x,y,z)\mapsto (xyz)_{ab}$, and that it forms a torsor with this law; the
symmetry law implies that $\cU_{ba}=\cU_{ab}$ as sets, but with torsor structures opposite to
each other. In particular, $\cU_a=\cU_{aa}$ is commutative. 
Associativity now says that the map
\[
\cU_{ab} \times \cU_{ab} \to \Bij(\cX), \qquad (x,z) \mapsto M_{xz}^{ab}
\]
is an inversive  action of $U_{ab}$ on $\cX$, and hence, by Lemma \ref{LR-Lemma},
 we have associated left and right  actions of the torsor $\cU_{ab}$ on $\cX$ given by
\begin{equation} \label{eqn:LR-def}
L_{xy}^{ab}:= M_{xu}^{ab} \circ M_{uy}^{ab}, \qquad
R_{yz}^{ab}:= M_{uy}^{ab}\circ M_{zu}^{ab} .
\end{equation}
Spaces with associative structure map form a category in the obvious way.
Many categorial notions can be defined exactly as in the case of Jordan structure
maps, see \ref{ssec:cat} above. The most important difference is that now, at several places, we have
to distinguish between ``left'' and ``right'': 
besides the inner ideals (intrinsic subspaces), we also have {\em left and right ideals}, that is,
subspaces $\cY$ that are invariant under the left, resp.\ right actions of the torsors
$U_{ab}$, whenever $a,b \in \cY$;
and besides homomorphisms, we also have {\em antihomomorphisms} (see \cite{BeKi2}). 

\begin{lemma}[The associative-to-Jordan functor]\label{la:asstojo}
If $(\cX,\top,M)$ is a geometry with associative structure map, then $(\cX,\top,J)$
is a geometry with Jordan structure map, where  
$$
J^{ab}_x:= M^{ab}_{xx},
$$
and the correspondence $(\cX,M) \mapsto (\cX,J)$ is functorial.
\end{lemma}

\begin{proof} Easy check of definitions.
%: the Jordan axioms are an immediate consequence.
\end{proof}

\nin 
A {\em special Jordan structure map} is the restriction of the map defined by the lemma
to some subspace $\cY \subset \cX$ which is stable under $J$. 
The following result says that all special Jordan geometries are subgeometries of some Grassmann geometry
(for Lagrange geometries, this is obvious; 
using Clifford algebras,
one can show that the structure map defined for
projective quadrics (cf.\   \ref{sssec:Quadrics}) is also special).

\begin{theorem}[Associative Grassmannian geometry] \label{th:GrasAs}
Let $\cX = \Gras(W)$ be the Grassmannian of an $\bA$-module $W$ with the transversality relation described in section
\ref{ssec:Grass-0}. For $(x,a) \in \cD_2$, let $P^a_x:W \to W$ be the $\bA$-linear projector with image $x$ and kernel $a$,
and, for $(x,a,z,b) \in \cD_4'$, define a linear operator on $W$ by
$$
M_{ab}^{xz}(y)=(P^a_x - P^z_b)(y).
$$
Then $M:\cD_4' \times \cX \to \cX$, $(x,a,z,b;y) \mapsto M_{ab}^{xz}(y)$ is an associative structure map. % on $\cX$.
\end{theorem}

\nin
The (elementary) proof has been given in \cite{BeKi1}. 

\begin{corollary}[Jordan Grassmannian geometry]
The formula 
$$
J_{a}^{xz}(y)=(P^a_x - P^z_a)(y)
$$
defines a Jordan structure map $J$  on the Grassmannian geometry.
\end{corollary}

\subsection{Self-dual geometries, and link with associative algebras}
A geometry with associative structure map
 is called {\em (strongly) self-dual} if it contains a closed transversal triple
$(a,b,c)=(o,o',e)$. Then
let $V:=V_{o'}$, $V' := V_o$ and $V^\times := U_{oo'} = V \cap V'$; this set is a group
with origin $e$ and group law
\begin{equation}\label{eqn:Grplaw}
xz = (xez)_{oo'} =  M_{xz}^{oo'}(e)=L_{xe}^{oo'}(z) = R_{ez}^{oo'}(x) .
\end{equation}
The left translation operator $L^{ab}_{xy}$ defined by (\ref{eqn:LR-def}) maps
$a$ to $a$ and $b$ to $b$, hence defines by restriction affine bijections  of $U_a$, resp.\ of
$U_b$, onto itself, and hence the group law defined by (\ref{eqn:Grplaw}) is $\Z$-bilinear with respect to
the  arguments $x$ and $z$. 
Under a regularity assumption,  this group law extends to an associative 
algebra structure on $V$  (Theorem \ref{th:AA}).%, and every associative algebra arises in this way, see \cite{BeKi1}). 

%example: $\Z\PP^1$ 

%say sth about associative idempotents and Morita contexts?

\section{Scalar action and major dilations}\label{sec:Scal}

From now on, we fix a commutative unital  {\em base ring} $\K$, with unit  denoted by $1$. 
%We do not yet assume that the scalar $2$ is invertible in $\K$.

\begin{definition}
Let $(\cX,\top,J)$ be a geometry with Jordan structure map.
A {\em $\K$-scalar action} on these data is given by  a structure map $S$, also called a
{\em scaling map},
\begin{equation}\label{eqn:Scal}
S  : \K  \times \cD_3   \to \cX, \qquad (r;y,a;x) \mapsto S^r_{y,a,x} =:r_y^a(x) 
\end{equation}
such that,  for every pair $(y,a)\in \cD_2$, the set $\cU_a$
 is turned into a $\K$-module with origin $y$,
 underlying abelian group structure given by
$x + z = J^{xz}_a(y) = (xyz)_a$, and  scalar multiplication given  by
$$
\K \times \cU_a \to \cU_a, \quad (r,z) \mapsto r_y^a (z) .
$$
Moreover, the scaling map (\ref{eqn:Scal}) shall extend  for invertible scalars to a global scaling map
\begin{equation}\label{eqn:Scal2}
S  : \K^\times \times \cD_2 \times \cX  \to \cX, \qquad (r;y,a;x) \mapsto S^r_{y,a,x} =:r_y^a(x) 
\end{equation}
such that the following properties hold: 

\begin{enumerate}
\item[(C)]
{\em compatibility:} the maps given by (\ref{eqn:Scal}) and (\ref{eqn:Scal2}) coincide on their common domain of definition,
%for $x,y \in U_a$, $r_y^a(x)= r x$ is multiplication by the scalar $r$ in $(U_a,y)$,
\ssk
\item[(A)]
{\em associativity:}   for $y$ fixed,
the map $\K^\times \times \cX \to \cX$, $(r,x) \mapsto r_y^a(x)$ is an action:
for all $r,s \in \K^\times$, we have $r_y^a \circ s_y^a = (rs)_y^a$ and $1_y^a=\id_\cX$,
\ssk
\item[(Du)]
{\em duality:}
%% = fundamental identity (PG1) from Be02; (PG2) disappears in the setting chosen here!
$(r\inv)_y^a  = r_a^y$
\ssk
\item[(Di)]
{\em distributivity:}
$r_y^a$ is an automorphism of $J$: $r_y^a \circ J^{xz}_b \circ (r_y^a)\inv =  
J^{r_y^a x,r_y^a z}_{r_y^a b}$, and similarly,
$J^{xz}_a$ is an automorphism of $S$:
$J^{xz}_a \circ r_y^b \circ J^{xz}_a = r_{J^{xz}_a(y)}^{J^{xz}_a(b)}$,
 % (translations preserve barycenters):
 \ssk
\item[(Tr)]
{\em link with translations:}
$(-1)_x^a = J_x^{aa}$ and
$r_x^a (r_y^a)\inv = L_a^{x,r_y^a x} = L_a^{r_x^a y,y}$.
%(implies the former (Af4))
\end{enumerate}
A {\em scaling map on an associative geometry} $(\cX,M)$ is simply a scaling map on the underlying Jordan geometry
$(\cX,J)$ of $(\cX,M)$.
\end{definition}

\begin{theorem}
Let $\cX = \Gras(W)$ be the associative Grassmannian geometry of an $\bA$-module $W$ (Theorem \ref{th:GrasAs}),
and $\K$ a unital ring contained in the center of $\bA$.
Then there is a scaling map on $\cX$ defined,
for $(y,a) \in \cD_2$ and $r\in \K$, by
$$
r_a^y (x) := (P^a_x + r P^x_a) (x) .
$$
\end{theorem}

\begin{proof} As to (\ref{eqn:Scal}),
the theorem describes the well-known affine space structure on the space of complements of $a$.
As to (\ref{eqn:Scal2}), the properties are easily checked  (cf.\ \cite{BeKi1}); 
for  the crucial property (Du), note that, for $r\in \K^\times$,
multiplying by the scalar $r\inv$ gives the same projective
 map, whence $r^a_y = r\inv P^a_y + P^y_a = (r\inv)_a^y$. 
 \end{proof}

\subsection{Affine algebra: major and minor dilations}
There are a lot of identies relating the ``major'' dilations $r_x^a$ with the ``minor'' 
dilations (translations $L_c^{yz}$). Most of them, such as (Tr), just rephrase and
globalize relations from usual affine geometry over $\K$ (cf.\ \cite{Be04}).
For instance, we can
change base points in $\cU_a$ by usual formulas from affine geometry:
if $o$ is an origin in $\cU_a$,  and $rx:=r_o^a(x)$ multiplication by $r$ 
in the $\K$-module $(\cU_a,o)$, then
\begin{equation}
r_y^a(x)= (1-r)  y + r  x.
\end{equation}
In the sequel, we will focus on the relation between scalar action and ``usual'' translations,
on the one hand, and ``quasi-translations'', on the other hand:
fix a base point $(o,o') \in \cD_2$; then the usual scalar action in the linear space
$(V,o)=(\cU_{o'},o)$ is given by $rv = r_o^{o'}(v)$, and the one in the linear space
$(V',o')=(\cU_o,o')$ by $ra = r_{o'}^o(a)$. 
For $v \in V$ we have by (Di)
\begin{equation}\label{eqn:ST1}
r_o^{o'} \circ L_{o'}^{vo} \circ (r_o^{o'})\inv = L_{o'}^{rv,o},
\end{equation}
which corresponds to the semidirect product structure of the usual affine group of $V$.
For $a \in V'$ we have,  by (Di) and (Du),
\begin{equation}\label{eqn:ST2}
r_o^{o'} \circ L_{o}^{ao'} \circ (r_o^{o'})\inv = L_{o}^{r\inv a,o'},
\end{equation}
which means that the ``quasi-translation'' $x^a:=L_o^{ao'}(x)$ for $x \in V$, $a \in V'$
satisfies the homogeneity relation $(rx)^a = r x^{ra}$.

\subsection{Midpoints, and generalized projective geometries} \label{rk:midpoints}
Assuming that  $2$ is invertible in $\K$, {\em midpoints in the affine space $\cU_a$} 
\begin{equation}\label{eqn:Midpt}
\mu (y,a,x) :=
(2_y^a)\inv (x) = \frac{x+y}{2}
\end{equation}
have
been extensively used in \cite{Be02}:  relation (Tr) implies that
\begin{align}\label{eqn:Midpt2}
2_x^a (2_y^a)\inv  & =  L_a^{x,2_y^a(x)} = L_a^{x,2x-y}= L_a^{yx} ,
\cr
 J^{aa}_{\mu(x,a,z)}  & = 
J^{\mu (x,a,z), \mu (x,a,z) }_a  = (2_x^a)\inv J_a^{zz} (2_x^a)  \cr
& =
(2_x^a)\inv 2^a_{J_a^{zz}(x)} J_a^{zz}   = L_a^{x,z} J_a^{zz} = J_a^{xz} ,
\end{align}
so translations and inversions $J^{xz}_a$
can be expressed by major dilations.
Moreover, by (\ref{eqn:Midpt2}),  every inversion is of the form
$J^{aa}_v$ for some $(a,v) \in \cD_2$; it follows that
(if $\cX$ is connected) all inversions
$J^{xz}_a$ are conjugate to each other under $\bfG$. 
The concept of {\em generalized projective geometry} (\cite{Be02}) is entirely based on scaling maps,
by assuming that $2$ is invertible in $\K$. In loc.\ cit.,
property (Du) appears as  ``Fundamental Identity (PG1)''; 
the identity (PG2) from loc.\ cit.\ does not appear in the axiomatics given here since it
concerns possibly non-invertible maps.

\begin{theorem}\label{th:gpg}
 Assume $2$ is invertible in $\K$.
If $(\cX,\top)$ is a generalized projective geometry, with scalar action
denoted by $r_x^a$ for $r \in \K^\times$, $x \top a$, then the map $J$  given by the following definitions,
is a Jordan structure map:  
\begin{equation*}\label{eqn:LaSym}
J^{aa}_x := (-1)_x^a, \qquad 
J^{xz}_a := (-1)^a_{\mu(x,a,z)} = J_{\mu (x,a,z)}^{aa} 
\end{equation*}
\end{theorem}

\begin{proof}
In the theorem, and in the proof, we suppress the superscripts $\pm$ used in \cite{Be02}
(formally, this can be justified  by working
in the ``dissociation'' of the geometry $(\cX^+,\cX^-)$). 
Using this notation, we check the defining identities of $J$:
Involutivity follows from $(-1)^2 = 1$, commutativity from the fact that
$\mu(x,a,z)=\mu(z,a,x)$, symmetry from the ``fundamental identity''
$(r_x^a)\inv = r_a^x$ (which implies $(-1)_x^a = (-1)_a^x$), 
distributivity holds since all maps $s_x^a$ for $s \in \K^\times$ are automorphisms of the scalar action map $\rr$,
and idempotency follows from the following computation in the affine space $U_a$:
$$
J^{xz}_a(y) = (-1)^a_{\frac{x+z}{2}} (y) = 2 \frac{x+z}{2} - y = x-y+z .
$$
Associativity is proved by establishing first that, in a generalized projective geometry, for all $x,y,z \top a$,
with the usual torsor structure $x-y+z$ on $U_a$,
$$
(-1)_a^x \circ (-1)_a^y \circ (-1)_a^z = (-1)_a^{x-y+z} .
$$
This identity is not among the defining identities given in \cite{Be02}, but it follows by combining the
``translation identity'' (T) from loc.\ cit.\ with the properties of scalar actions.
Using this, associativity follows in a straightforward way: 
\begin{align*}
J_a^{xz}J_a^{uv} J_a^{pq} &=
(-1)_a^{\mu(x,a,z)} \circ (-1)_a^{\mu(u,a,v)} \circ (-1)_a^{\mu(p,a,q)}
=  (-1)_a^{ \frac{ x+ z}{2}  - \frac{ u+v }{2}  + \frac{p +q}{2}  } 
\cr
& = (-1)_a^{\frac{(x-v+p) + (z-u+q)   }{2}} = J_a^{ J_a^{xp}(v),J_a^{zq}(u)} \, .
\end{align*} 
\end{proof}

\subsection{Remark on the base ring $\Z$}
A geometry with  Jordan structure map $J$ always carries a $\Z$-scalar action:
indeed, an abelian group $(\cU_a,y)$ is automatically a $\Z$-module, and since
$\Z^\times = \{ \pm 1 \}$, the scaling  map (\ref{eqn:Scal2})  can be defined by
letting $1^a_y =\id_\cX$ and $(-1)^a_x = J^{aa}_x$. It is easily checked that this satisfies the
properties (C) through  (Tr). 
Moreover, by (Tr), any $\Z$-scalar action is necessarily given by these formulae.
Thus a geometry with Jordan structure map is the same as one with compatible $\Z$-action.

\section{Idempotents and the modular group}\label{sec:Id}

Let $\cX$ be a geometry with Jordan structure map $J$.
By a {\em configuration of points in $\cX$} we just mean a subset $P \subset \cX$.
In this chapter, we study some simple configurations:

\begin{enumerate}
\item
$P=\{ x,a \}$, with $(a,x) \in \cD_2$ (transversal pair),
\item
$P = \{ o , a, z \}$, with $(o,a,z) \in \cD_3$ (transversal triple), but not closed, % (not in $\cD_3'$),
\item
$P = \{ a,b,c \}$, where $(a,b,c) \in \cD_3'$ (pairwise transversal triple),
\item
$P = \{ a,x,b,y \}$, where $(a,x,b,y)$ is an idempotent quadruple. % (see def.\ below).
\end{enumerate}

\nin
For any configuration, consider the ``group generated by inversions from $P$''
\begin{equation}\label{eqn:Generated}
\bfG_{(P)} := \Bigl\langle J^{xz}_a \mid \, x,a,z \in P, \, x\top a, z \top a \Bigr\rangle \subset \Aut(\cX)
\end{equation}
and the smallest subgeometry $\langle P \rangle \subset \cX$ containing $P$.
For configuration (1), $\bfG_{(P)} = \{ J^{aa}_x,\id \}$ is isomorphic to $\Z/2\Z$;
for configuration (2), $\bfG_{(P)}$ contains a subgroup that is a quotient of $\Z$,
generated by $L^{zo}_a = J^{zo}_a \circ J^{oo}_a$, and the whole group
$\bfG_{(P)}$ is a quotient of $\Z \ltimes (\Z/2\Z)$. 
Then $\langle P \rangle$ is a flat geometry (see \ref{ssec:flat}). 
Configuration (3) is more interesting: 

\begin{theorem}
Assume that $(a,b,c)$ is a pairwise transversal triple. Then 
\begin{equation} \label{eqn:S3}
\bfS := 
\bigl\{
\id_\cX, J^{ab}_c, J^{ac}_b,J^{cb}_a, J^{ab}_c \circ J^{ac}_b, J^{ac}_b \circ J^{ab}_c 
\bigr\} 
\end{equation}
is a subgroup of $ \Aut(\cX,\top,J)$, isomorphic to the permutation group $S_3$.
\end{theorem}

\begin{proof}
We claim that the following correspondences are group homomorphisms:

\bigskip
%\begin{center}
\nin
\begin{tabular}{ l  | l   | l   | l  | l   | l }
$(1)$ & $(12)$ & $(23)$  & $(13)$ & $(123)=(12)(23)$  & $(132)=(23)(12)$  
\\
%\hline
[3 mm]
$\id_\cX$ & $J^{ab}_c$ & $J^{bc}_a$ & $J^{ac}_b$ & 
$C_{abc}:=J^{ab}_c \circ J_a^{bc}$ & $C_{bac}=J^{bc}_a \circ J^{ab}_c$
\\ [3 mm]
$1$ & $\begin{pmatrix} 0 & 1 \\  1& 0 \end{pmatrix}$ &
$\begin{pmatrix}  1 & 0  \\ 1 & -1 \end{pmatrix}$ &
$\begin{pmatrix}  -1& 1 \\ 0 & 1 \end{pmatrix}$ &
$\begin{pmatrix}  1& -1 \\ 1 & 0 \end{pmatrix}$ &
$\begin{pmatrix}  0 & 1 \\ -1  & 1 \end{pmatrix}$ 
\\ [5 mm]
$z$ & $z\inv $ & $(1-z\inv)\inv $ & $1-z$ &    
$1-z\inv $ & $(1-z)\inv $
\end{tabular}
%\end{center}

\bigskip
In this table,  
we list  the elements of $S_3$ first,  then the corresponding element of $\bfS$, 
a corresponding element of $\PP\Gl(2,\Z)$, and the fractional linear transformation (in the 
variable $z$)  corresponding
to the element from the precedig line. 
Indeed,
it is checked by direct computation that these correspondences are group homomorphisms:
 since the elements $J^{ab}_c,J^{ac}_b,J^{cb}_a$ are of order two, it suffices to show that the
composition of any two of them is a $3$-cycle, e.g., that
$( J^{ab}_c \circ J^{ac}_b )^3 = \id_\cX$:
\begin{equation} \label{eqn:S3b}
( J^{ab}_c \circ J^{ac}_b )^3 = \bigl( J^{ab}_c J^{ac}_ b  J^{ab}_c \bigr)
 \bigl( J^{ac}_ b  J^{ab}_c J^{ac}_ b \bigr) = J_a^{bc} J_a^{bc} = \id_\cX 
\end{equation}
by using (IN), (IP), (D), and (C). \end{proof}

\begin{remark}
If $\cX$ is the projective line over $\K=\Z/2\Z$, then $\bfG = \bfS = S_3$.
\end{remark}

\begin{remark}
The action of matrices from $\Gl(2,\Z)$ 
defined by this and the following tables corresponds to its
``usual'' action by fractional linear transformations on a Jordan algebra with unit $1$, 
as indicated.
See Section \ref{sec:JA} for more on this.
\end{remark}

\begin{theorem}\label{th:Modular1}
Assume that $(a,b,c)$ is a pairwise transversal triple and $P=\{ a,b,c \}$.
Then $\bfG_{(P)}$ is a quotient of $\PP \Gl(2,\Z)$.
More precisely,  define the matrices 
\begin{equation}\label{eqn:matrices}
S:= \begin{pmatrix} 0 & 1 \cr -1 & 0\end{pmatrix}, \,
T:= \begin{pmatrix} 1 & 1 \cr 0 & 1 \end{pmatrix}, \,
F :=\begin{pmatrix} 0 & 1 \cr 1 & 0 \end{pmatrix}, \,
I := \begin{pmatrix}1 & 0 \cr 0 & - 1 \end{pmatrix} \in \Gl(2,\Z) \, ,
\end{equation}
and by $[S]$, etc.,  denote the corresponding element in $\PP\Gl(2,\Z)$. 
Then there is a unique group epimorphism
$$
\phi: \PP\Gl(2,\Z) \to \bfG_{(P)} , 
% \phi: \PP \SL(2,\Z) \to G(\cX), 
$$
defined by the correspondences $[S]\mapsto J^{bb}_a J^{ab}_c$ and
 $[T]\mapsto  L^{ca}_b = J^{ca}_b J^{aa}_b$
and $[I] \mapsto  J^{bb}_a$.
Moreover,
we then have the following correspondences (notation as above):

\bigskip
\begin{center}
\begin{tabular}{ l  | l   | l   | l  | l   | l }
 $J^{aa}_b$ & $J^{bb}_c$ & $J^{cc}_a$ & 
$L^{ba}_c=J^{ba}_c  J_c^{aa}$ & $L^{ca}_b=J^{ca}_b  J^{aa}_b$ &
$L_a^{bc} = J_a^{bb}J^{bc}_a$ 
\\ [3 mm]
 $\begin{pmatrix} -1 & 0 \\  0& 1 \end{pmatrix}$ &
$\begin{pmatrix}  -1 & 2  \\ 0 & 1 \end{pmatrix}$ &
$\begin{pmatrix}  -1& 0 \\ -2 & 1 \end{pmatrix}$ &
$\begin{pmatrix}  1& -1 \\ 1 & 0 \end{pmatrix}$ &
$\begin{pmatrix}  1 & 1 \\  0  & 1 \end{pmatrix}$ &
$\begin{pmatrix}  1 & 0 \\  -1   & 1 \end{pmatrix}$ 
\\ [5 mm]
$-z$ & $2-z$ & $(2- z\inv)\inv $ & $(z\inv -1)\inv $ &    
$z + 1$ & $(1-z\inv)\inv $
\end{tabular}
\end{center}
\end{theorem}

\msk

\begin{proof}
Recall that
 the modular group
$\Gamma := \PP\SL(2,\Z)$ is presented by generators and relations
\begin{equation}\label{eqn:GR} \Gamma = 
\Big\langle  [S]  , \, [T]   \quad \mid \quad
[S]^2 = 1, \, [ST]^3 = 1 \Big\rangle \, .
\end{equation}
We prove the relation corresponding to $[S]^2 = 1$, that is, $( J^{bb}_a J^{ab}_c )^2 = \id$:
$$
(J^{bb}_a J^{ab}_c)^2 = J^{bb}_a (J^{ab}_c J^{bb}_a J^{ab}_c) =
J^{bb}_a J^{aa}_b = (J^{bb}_a)^2 = \id .
$$
Next, we prove the relation corresponding to $[ST]^3 = 1$:
note that $J^{bb}_a J^{ab}_c J^{ca}_b J^{aa}_b =
J^{bb}_a J^{ab}_c J^{ca}_b J^{bb}_a =
J^{ab}_c J^{cb}_a$, and according to (\ref{eqn:S3b}), this is a $3$-cycle.
According to the presentation (\ref{eqn:GR}), this defines a homomorphism $\PP\SL(2,\Z) \to \Aut(\cX)$.
The remaining correspondences are checked by similar computations, and they establish
a homomorphism $\PP\Gl(2,\Z) \to \bfG_{(P)}$ which is obviously surjective.
\end{proof}

\begin{theorem}\label{th:Modular2}
Recall from \ref{ssec:Grass-0} the definition of the projective line
 $\Z \PP^1=\Gras_1^1(\Z^2)$. 
We denote its canonical pairwise transversal triple by
$o=[e_1]$, $\infty = [e_2]$, $e=[e_1+e_2]$.
Assume $(a,b,c)$ is  a pairwise transversal triple in $\cX$ and $P = \{ a,b,c \}$.
Then the geometry $\langle P \rangle$ is a quotient of $\Z \PP^1$. More precisely, there is a
unique morphism of geometries 
$$
\Phi : \Z\PP^1 \to \langle P \rangle
$$
which preserves the pairwise transversal triples:
$\Phi (o) = a$, $\Phi(\infty)=b$, $\Phi(e)=c$.
This map 
 is equivariant with respect  to the homomorphism $\phi$ from the preceding theorem in the sense
 that $\Phi (g.x)=\phi(g) \Phi(x)$ for all $g \in\PP \Gl(2,\Z)$.
\end{theorem}

\begin{proof}
The projective line $\Z \PP^1$ is homogeneous under the
group $\PP \Gl(2,\Z)$. 
As base point in the set $\cD_2(\Z \PP^1)$ of transversal pairs we take
$(o,\infty)=([e_2],[e_1])$. The stabilizer $\bfH$ of this pair in $\Gl(2,\Z)$ is the group of diagonal matrices.
Since $\phi(I)=J^{bb}_a$, and $J^{bb}_a$ preserves the pair $(a,b)$, the map $\phi$
from the theorem induces a well-defined and base point preserving map 
$\Phi_2 : \cD_2(\Z\PP^1) \to \cD_2(\cX)$.
Let $\pr_1 : (x,a) \mapsto x$ the projections from $\cD_2$ to $\Z \PP^1$ and to $\cX$,
respectively.
Since the group $\Gl(2,\Z)$ and its image group under $\phi$ preserve the respective
transversality relations, there is a well-defined map $\Phi :\Z\PP^1 \to \cX$ such that
$\pr_2 \circ \Phi_2 = \Phi \circ \pr_2$.
It maps $o$ to $a$, and by equivariance, it maps  $\infty$ to $b$ and $e$ to $c$.
As a consequence of the equivariance property of $\Phi$, it follows that $\Phi$
 is a morphism of geometries, i.e., we have
$\Phi(J^{uv}_w (y)) = J^{\Phi(u)\Phi(v)}_{\Phi(w)} \Phi(y)$ whenever defined.
\end{proof}

\begin{comment}%%%%
$\cD_2(\Z \PP^1)$ is homogeneous under $\PP \SL(2,\Z)$, and the stabilizer of the canoncal base point
is trivial !! (this occurs only for $\Z$) 
Thus it has a natural torsor structure.
Moreover, this torsor structure can be expressed in terms of $J$, namely by the $\La$-operators;
therefore it should play a role... 
ˆ suivre !!
\end{comment}%%%%%%%

\begin{remark}
Of the many relations that are valid in the setting of the preceding theorems,
 let us just mention the following: the involution $J^{ca}_b$ has, besides $b$, another 
fixed point given by $J^{aa}_c(b)$:
\begin{equation}\label{eqn:FP!}
J^{ca}_b (J^{aa}_c (b)) = J^{aa}_c(b) .
\end{equation}
Indeed, $J^{ca}_b J^{aa}_c (b) = J^{ca}_b J^{aa}_c J^{ca}_b(b)= J^{cc}_a(b)= J^{aa}_c (b)$.
%The reader is invited to check this by a matrix computation; he may find that computations are more
%transparent using our $J$-operators.
Another non-trivial relation is
\begin{equation}\label{eqn:JP!}
J^{ac}_b = J^{ac}_{J^{aa}_c(b)},
\end{equation}
coming from $J^{aa}_c J^{ac}_b J^{aa}_c J^{ac}_b = 1$.
In order to get a visual image of such and other relations, the best realization of $\Z \PP^1$ is not a ``line'' but
rather  a tesselation of the hyperbolic plane of type $(2,3,\infty)$; such images can be found on the internet,
see e.g., 
\url{http://upload.wikimedia.org/wikipedia/commons/thumb/0/04/H2checkers_23i.png/1024px-H2checkers_23i.png}.
In this image, the points $a,b,c$ may be chosen as points on the boundary 
circle such that the triangle $(a,b,c)$ contains
as its ``center'' a point
of rotational symmetry with order $3$. The symmetries $J^{ab}_c$ are then easily visible, but the 
orbit of $a,b,c$ (the set $\langle P \rangle$) will be on the boundary circle;  thus
this visualisation gives only a partial image, but at least
it may give an idea of how complicated the corresponding geometry really is. 
In particular, 
the orbits of the translations groups defined by $a,b$, resp.\ $c$ correspond to 
limits of $\Z$-points of horocycles touching
the boundary circle at $a,b$, resp.\ $c$. 
\end{remark}

The projective line $\Z\PP^1$ and its quotients are the most elementary building blocks for
analyzing the structure of a general geometry.
It is important that $\Z \PP^1$ appears not only in the context of a self-dual geometry,
where pairwise transversal triples exist, but also for certain geometries ``of the second kind'',
namely those having {\em idempotents}.
The following definition arises when retaining the properties of the quadruple
$(a,b,c,a)$, where $(a,b,c)$ is pairwise transversal, but then allowing some pairs 
 to be not necessarily transversal:

\begin{definition}
 We say that $(a,x,b,y) \in \cD_4$  is an {\em idempotent} if it satisfies
\begin{equation}\label{eqn:Idem1}
J^{aa}_x(y)= y,  \, \,
J^{xy}_b J^{aa}_x(b)= J^{aa}_x(b), \, \,
J^{aa}_x J^{xy}_b(a) = J^{xy}_b(a) , \, \,  J_b^{xy} J^{ab}_x(y)=J^{ab}_x(y)  
\end{equation}
\begin{equation}\label{eqn:Idem2}
J^{bb}_y(a)= a,  \, \,
J^{ab}_x J^{yy}_b(x)= J^{yy}_b(x), \, \,
J^{yy}_b J^{ab}_x(y) = J^{ab}_x(y) ,  \,  \, J^{ab}_x J^{xy}_b(a)=J^{xy}_b(a) .
\end{equation}
A {\em strong idempotent} is an idempotent $(a,x,b,y)$ such that, moreover,
\begin{equation}\label{eqn:Strong}
J^{yx}_{J_x^{aa}(b)} = J^{a,J_b^{xy} (a)}_{J^{ab}_x(y)}    \, .
\end{equation}
\end{definition}

\begin{lemma} If $(a,b,c)$ is a pairwise transversal triple, then
$(a,x,b,y):=(a,c,b,a)$ is a strong idempotent.
\end{lemma}

\begin{proof} Easy check -- cf.\ 
(\ref{eqn:FP!}); the ``strong'' relation (\ref{eqn:Strong}) boils down to (\ref{eqn:JP!}).
\end{proof}

\nin Conditions (\ref{eqn:Idem1}) and (\ref{eqn:Idem2}) are dual to each other in the sense that they
imply that $(a,x,b,y)$ is an idempotent if and only if so is $(y,b,x,a)$.
Another way to formulate this definition is to define 4 new points
\begin{equation}
c:=J^{aa}_x(b), \quad
z:= J^{yy}_b(x), \quad
d:= J^{xy}_b(a), \quad
w:= J^{ab}_x(y),
\end{equation}
(thinking of $(x,y,z,w)$ and $(a,b,c,d)$ as two harmonic quadruples on two
 ``dissociated'' projective lines $\ell, \ell'$, in the sense of \ref{ssec:Diss})
and to require that
\begin{equation}
J^{aa}_x(y)= y,  \,
J^{xy}_b (c)= c, \,
J^{aa}_x (d) = d, \, 
J^{xy}_b(w)=w \, ,
\end{equation}
\begin{equation}
J^{bb}_y(a)= a,  \,
J^{ab}_x (z)= z , \,
J^{yy}_b  (w)=w, \, 
J^{ab}_x(d)=d   \, .
\end{equation}
Geometrically, this means that  certain fixed points of our involutions on the lines $\ell, \ell'$
are determined in a definite way. Fixed points of $J^{bb}_x$ are then given by
\begin{equation}
J^{bb}_x (w) =J^{bb}_x J^{ab}_x (y)=  J^{ab}_x J^{aa}_x (y)=w, \qquad J^{xx}_b (d)=d .
\end{equation}

\begin{theorem}\label{th:Idempot}
Assume $(a,x,b,y) \in \cD_4$ is a strong idempotent.
Then there is a homomorphism $\Gl(2,\Z) \to \Aut(\cX)$ defined by the following correspondences:

\begin{center}
\begin{tabular}{ l  | l   | l   | l  | l   | l   l l  l l  l   | l   | l   | l  l l}  
 $J^{bb}_x$ & $J^{xy}_b$ & $L^{yx}_b$ & 
$J^{bb}_y$ & $L^{ab}_x$ &
$J_x^{ab}$ 
\\ [3 mm]
 $\begin{pmatrix} -1 & 0 \\  0& 1 \end{pmatrix}$ &
$\begin{pmatrix}  -1 & 1  \\ 0 & 1 \end{pmatrix}$ &
$\begin{pmatrix}  1 & 1 \\  0 & 1 \end{pmatrix}$ &
$\begin{pmatrix}  -1& 2 \\ 0 & 1 \end{pmatrix}$ &
$\begin{pmatrix}  1 & 0 \\  1  & 1 \end{pmatrix}$ &
$\begin{pmatrix}  -1 & 0 \\  -1   & 1 \end{pmatrix}$ 
\\ [8mm]
%\end{tabular}
%\end{center}
%\bigskip
%\begin{center}
%\begin{tabular}{ l  | l   | l   | l  | l          }
$J^{aa}_x$ &
 $\La^{ab}_{xy}=L^{ab}_x L^{xy}_b   $ & $(\La^{ab}_{xy})^3$ & $ W_{ab}^{xy}= L_x^{ab} L_b^{xy} L_x^{ab}$ & 
$J_b^{xy} J^{aa}_x$ & $(J_b^{xy} J^{aa}_x)^2$ 
\\ [3 mm]
$\begin{pmatrix} -1 & 0 \\ -2 & 1 \end{pmatrix}$ &
 $\begin{pmatrix} 1 & -1 \\  1& 0 \end{pmatrix}$ &
$\begin{pmatrix}  -1 & 0  \\ 0 & -1 \end{pmatrix}$ &
$\begin{pmatrix}  0 & -1 \\  1 & 0 \end{pmatrix}$ &
$\begin{pmatrix}  -1& 1 \\ -2 & 1 \end{pmatrix}$ &
$\begin{pmatrix}  -1 & 0 \\  0  & -1 \end{pmatrix}$ 
\end{tabular}
\end{center}
\msk
If $(a,x,b,y)$ is an idempotent (not necessarily strong), then a similar statement still holds, but
$\Gl(2,\Z)$ has to be replaced by the universal central extension $\widetilde{\Gl(2,\Z)}$
(that is, by an extended braid group). 
\end{theorem}

\begin{proof} 
Let $P = \{ a,x,y,b \}$.
The group $\bfG_{(P)}$ is clearly generated by the three elements
$A:=L_x^{ab}$, $B:=L_b^{xy}$ and $J:= J_x^{bb}$.
We show that these elements satisfy the following relations defining  $\Gl(2,\Z)$ 
%(see 
%\url{https://en.wikipedia.org/wiki/Presentation_of_a_group} for a presentation of $\Gl(2,\Z)$, print
%reference \ref ? ):
\[
(ABA)^4 = 1, \quad ABA = BAB, \quad J^2 = 1, \quad (JA)^2 = 1 = (JB)^2 .
\]
Indeed, the proof of the last three relations is immediate. 
In order to prove the first relation,
we start by proving that the following element $Z$ is central in $\bfG_{(P)}$:
\begin{equation}
Z:= (J_b^{xy} J^{aa}_x)^2 =  J_b^{xy} J^{aa}_x  J_b^{xy} J^{aa}_x =
J_b^{xy} J^{x, J_x^{aa}(y)}_{J^{aa}_x(b)} =
J_b^{xy} J^{xy}_{J^{aa}_x(b)} \, .
\end{equation}
It is obvious that  $Z(x)=x$ and $Z(y)=y$; using (\ref{eqn:Idem1}), it follows that also
$Z(a)=a$ and $Z(b)=b$. 
Therefore $Z$ commutes with all generators $J^{uv}_w$ of $\bfG_{(P)}$:
$Z J^{uv}_w Z\inv = J^{Zu,Zv}_{Zw} = J^{uv}_w$, and hence is central in $\bfG_{(P)}$.
Moreover, $Z$ is of order $2$, since
$$
J_b^{xy} J^{xy}_{J^{aa}_x(b)} J_b^{xy} = J^{xy}_{J_b^{xy} J^{aa}_x(b)} = J^{xy}_{J^{aa}_x(b)} ,
$$
and hence $Z$ is a product of two commuting involutions.
Now consider the ``Weyl-element'' (cf.\ \cite{Lo95}, 6.1)
$W:=W_{ab}^{xy}= ABA =L_x^{ab} L_b^{xy} L_x^{ab} = J_x^{ab} \, J_b^{xy} J_x^{aa} \, J_x^{ab}$. The last 
expression shows that $W$  is conjugate to $J_b^{xy} J_x^{aa}$, and hence $W^2$ is conjugate to
$(J_b^{xy} J_x^{aa})^2 = Z$.
Since $Z$ is central, it follows that $W^2 = Z$, and so $W^4 = Z^2 = \id_\cX$.

\begin{comment}%%%% here old ``brute-force computation'' :
Let us show that $W^4 = \id$.
Conjugating by $J_x^{ab}$, it suffices to show that
$W':=  J_b^{xy} J_x^{aa}$ is of order $4$.
Now,
$$
(W')^4 = (  J_b^{xy} J_x^{aa} )^4 =
 J_b^{xy} J_x^{aa} J_b^{xy} J_x^{aa} J_b^{xy} J_x^{aa} J_b^{xy} J_x^{aa} =
 J_b^{xy} \, J^{  J_x^{aa} J_b^{xy} J_x^{aa}(x), J_x^{aa} J_b^{xy} J_x^{aa} (y)}_{ J_x^{aa} J_b^{xy} J_x^{aa} (b) }
 $$
But the idempotency relations imply that
$J_x^{aa} J_b^{xy} J_x^{aa}(x)=J_x^{aa}(y)=y$,  and
$ J_x^{aa} J_b^{xy} J_x^{aa} (y)= J_x^{aa} (x) =x$,  and
$ J_x^{aa} J_b^{xy} J_x^{aa} (b)=J_x^{aa} J_x^{aa}(b)=b$,
whence $(W')^4 = J_b^{xy} J_b^{yx} = \id$.

$W^2$ central: has to commute with all generators !
(in fact, it will act trivially on the projective line generated by our given points, but not on the whole of $\cX$:
this will give the Peirce-spaces).
\end{comment}%%%%%%%%%%

\ssk
Next, we prove that $Z':= (AB)^3$ is a central element of order $2$.
Indeed, the proof is very similar to the one given above: we have
\begin{equation}
Z' = J_x^{ab} J_b^{xy}  J_x^{ab} J_b^{xy}  J_x^{ab} J_b^{xy}  = 
J_x^{ab} J_{J_b^{xy} (a)}^{y, J^{ab}_x(y)}   
\end{equation}
As above, it is checked that $Z'$ fixes $a,x,b,y$, and hence is central; it is a product of two commuting
involutions, hence of order $2$ (and hence $(AB)^6 = 1$).

\ssk
Since $Z = W^2 =ABAABA$ and $Z'=ABABAB$, the relation $ABA = BAB$ is equivalent to $Z=Z'$ or to $ZZ'=1$. But, by an easy computation,
\begin{equation}
ZZ' = J^{yx}_{J_x^{aa}(b)} \, J^{a,J_b^{xy} (a)}_{J^{ab}_x(y)}
\end{equation}
so $ZZ'=1$ is equivalent to (\ref{eqn:Strong}). This proves the claim for a strong idempotent.
If the idempotent is not strong, then, as we have seen, all relations from $\Gl(2,\Z)$ a satisfied, possibly up to
central elements. Therefore the homomorphism may be defined on the level of the universal central 
extension (cf.\ \cite{St67}, \S 7, (ix), p.67).  
\end{proof}

\begin{remark}
\begin{comment}%%%%
The ``strange'' condition (\ref{eqn:Strong}) finds its explanation by the proof. It can be written in various forms
(and we are not sure if the form given here is the ``best'' one).
It is not clear neither whether this relation is a consequence of the other ones (i.e., whether an idempotent
is always strong); we have, however, the impression that this is not the case: Conditions (\ref{eqn:Idem1}), 
(\ref{eqn:Idem2})
describe fixed points ``near'' the initial points of our configuration; but the equality expressed by (\ref{eqn:Strong})
seems not to be reducible to such a condition. 
\end{comment}%%%%
It is  not true that the homomorphism always factorizes via $\PP\Gl(2,\Z)$:
the central element $Z$ (or $Z'$) acts trivially on the $\bfG_{(P)}$-orbit of $x,a,b,z$,
but in general it will act non-trivially on the whole of $\cX$ (cf.\ the following example)
 -- this action is precisely described by the
{\em Peirce-decomposition} associated to the idempotent.
In a similar way, the geometry $\langle P \rangle$ is not always a quotient of
$\Z \PP^1$, but rather of the {\em dissociation of $\Z \PP^1$}.
\end{remark}

\begin{example} Let $\cX = \Gras(W)$ be the Grassmannian geometry of a $\K$-module $W$, fix  a direct sum
 decomposition $W = E \oplus F \oplus H$, and (non-zero) subspaces $u,v,w \subset F$
 such that $u \oplus v = F = u \oplus w = v \oplus w$ (so $(u,v,w)$ is a pairwise transversal triple in
 $\Gras(F)$). Let
\begin{equation}
a:= w \oplus H, \quad
x:= E \oplus u, \quad
b:= H \oplus v, \quad
y:= E \oplus w .
\end{equation}
Then $(a,x,b,y)$ is a chain in $\cX$, but $a \cap y= w$, so $a$ and $y$ are not transversal.  It can be shown that
$(a,x,b,y)$ is a (strong) idempotent in $\Gras(W)$. Instead of checking the defining properties,
it is easier to exhibit directly the corresponding realization of $\Gl(2,\Z)$ in $\Aut(\cX)=\PP \Gl(W)$:
we decompose $W = E \oplus u \oplus v \oplus H$, and write elements of $\Gl(W)$ accordingly as
$4 \times 4$-matrices. Then, considering $w$ as diagonal in $u \oplus v$, all four middle blocks are
square matrices, so that a matrix
$\bigl(\begin{smallmatrix}
a&b\\ c&d
\end{smallmatrix} \bigr) \in \Gl(2,\Z)$ may be identified with the class of the matrix
\[
\begin{pmatrix} 1 & 0 & 0 & 0 \cr 0 & a & b & 0 \cr 0 & c & d & 0 \cr 0 & 0 & 0 & 1 \end{pmatrix} 
% \mbox{ in } \, \PP \Gl(W) \, .
\]
%The image of under this imbedding does not act trivially on the Grassmannian.
in $\PP\Gl(W)$.
\end{example}

\section*{SECOND PART:
%INFINITESIMAL CALCULUS, AND 
TANGENT OBJECTS}

Our aim in this part is to associate to a Jordan or associative structure map, at a given base 
point $(o,o')$, a ``tangent object'', namely a Jordan pair or algebra, resp.\ an 
associative pair or algebra. 
As explained in the introduction,
this requires additional {\em regularity assumptions}, and
 there  are different ways to
formalise  them; the way chosen here is via {\em algebraic differential calculus}
as developed in \cite{Be14}: we assume that the whole setup is
{\em functorial with respect to scalar extensions of $\K$ by Weil algebras $\bA$},
which implies that the geometries are {\em Weil manifolds}.

%%%%%%%%%%%%%%%%%%%%%%%%%%%%%%%%%%%%%%%%%%%

\section{Jordan geometries over $\K$}\label{ssec:S}

\subsection{Weil spaces and Weil manifolds}\label{ssec:Weils}
{\em Weil spaces} $\ul M$ generalize smooth manifolds in the sense that they have  {\em tangent bundles}
$M^\bA :=T^\bA M$, generalizing the classical bundles $TM, TTM$, etc. 
We recall some basic concepts from \cite{Be14} (see also \cite{Be08, KMS} on {\em Weil functors}).

\begin{definition}
A {\em $\K$-Weil algebra} is an associative and commutative algebra 
\begin{equation}
\bA = \K \oplus \mbA
\end{equation}
where $\mbA$ is a nilpotent ideal of $\bA$ which is free and finite-dimensional as 
a $\K$-module. 
Weil algebras form a category $\ul\Walg_\K$, where morphisms are algebra homomorphisms preserving decompositions. 
Note that projection $\pi:\bA \to \K$ and injection $\zeta: \K \to \bA$ are morphisms. 
Main examples of Weil algebras are the {\em jet rings}
\begin{equation}\label{eqn:jetring}
J^k \K := \K[X]/(X^{k+1}),
\end{equation}
which for $k=1$ give
the  {\em tangent ring of $\K$}, or {\em ring of dual numbers over $\K$} :
\begin{equation}
T\K :=\K[X]/(X^2) = \K[\eps]= \K \oplus \eps \K \qquad (\eps^2 = 0) .
\end{equation}
\end{definition}

\begin{definition}
A {\em Weil space} is a functor $\ul M$ from the category of $\ul\Walg_\K$ of $\K$-Weil algebras to the category
$\ul\set$ of sets, and a {\em Weil law} is a natural transformation $\ul f:\ul M \to \ul N$ of Weil spaces, that is,
 we have sets $M^\bA$ and maps $f^\bA$, varying functorially with $\bA$:
 a Weil algebra morphism $\phi:\bA \to \bB$ induces a map $M^\phi : M^\bA \to M^\bB$, and
 \begin{equation}\label{eqn:functional}
 f^\bB \circ M^\phi = N^\phi \circ f^\bA .
\end{equation}
We let  $M:=M^\bK$ and  $f:=f^\bK$.
The projection $\bA \to \bK$ induces a map $M^\bA \to M$, and equation (\ref{eqn:functional}) shows that
$f^\bA$ is fibered over the {\em  base map} $f$.
Similarly, the injection $z:\K \to \bA$ induces a {\em zero section }
$z^\bA:M \to M^\bA$. 

The notation $T^\bA M:=M^\bA$, $T^\bA f:=f^\bA$ is also used, and
 the set $TM:=M^{T\K}$ is called the
{\em tangent bundle of $M$} and the map $Tf:=f^{T\K}$ the {\em tangent map of $f$}.

A {\em flat Weil space} is given by a $\K$-module $V$ and $V^\bA := V \otimes_\K \bA$,
and $f^\bA$ the algebraic scalar extension of $f$ in case  $f:V \to W$ is a polynomial.
\end{definition}

Every concept defined in terms of the category $\ul\set$ allows for  a ``Weil'' counterpart,  by
taking the functor catogery of functors from $\ul\Walg_\K$ into that category:

\begin{definition}
A {\em Weil manifold}, modelled on a flat Weil space $\ul V$, is a $\K$-Weil space $\ul M$ together with set-theoretic atlasses
on $M^\bA$, for
each Weil algebra $\bA$ (cf.\  def.\ \ref{def:atlas})
$\cA^\bA = (U^\bA_i, \phi_i^\bA,V_i^\bA)$, modelled on $V^\bA$, and depending functorially
on $\bA$. 

\ssk 
A {\em Weil Lie group with atlas} $(\ul G, \ul m,\ul i, \ul e)$ is Weil manifold $\ul G$  together with group structures
on $G^\bA$ depending functorially on $\bA$. (We suppress the atlas in the notation; in \cite{Be14} we consider
more general group objects, without atlas.)

\ssk 
A  {\em Weil symmetric space} $(\ul M,\ul s)$ is a Weil manifold $\ul M$ together with 
reflection space structures $s^\bA : M^\bA \times M^\bA \to M^\bA$ (see definition \ref{ReflectionDef}), depending functorially
on $\bA$, and such that, moreover, for each $x \in M$, the tangent map
$T_x (s_x)$ is minus the identity map on the tangent space $T_x M$ (fiber of $TM$ over $x$):
\begin{equation} \label{eqn:-1}
\forall x \in M , \forall u \in T_x M : \quad s^{T\K} (x , u) = -u .
\end{equation}
\end{definition}

\subsection{Jordan and associative geometries over $\K$}

\begin{definition} \label{def:JGeo}
A  \emph{$\K$-Jordan geometry}  $(\ul \cX, \ul \top, \ul J, \ul S)$ is a $\K$-Weil space $\ul\cX$, together with families
$\top^\bA$ of transversality relations,
$J^\bA$ of Jordan structure maps, and $S^\bA$ of scaling maps, depending functorially on $\bA$, 
such that 
\begin{enumerate}
\item
$\cX$ is a Weil manifold with respect to the canonical atlas $\cA^\bA$ on $\cX^\bA$
defined for each Weil algebra $\bA$ by Lemma \ref{la:atlas},
\item
for all $(a,x,b)\in \cD_3^\K$,
the tangent map of $J^{ab}_x$ at its fixed point $x$ is $-\id_{T_x \cX}$:
$$
T_x(J_x^{ab}) = - \id_{T_x \cX}
$$
%(i.e., $\cU_{ab}^\K$ is a symmetric space in the sense of the preceding definition).
\end{enumerate}
Likewise, {\em associative geometries over $\K$} are defined, replacing $J$ by $M$.
{\em Morphisms} are the respective Weil laws that are compatible with the additional structures.
\end{definition}

Condition (1)  amounts to requiring that affine parts of $\cX^\bA$ are usual
algebraic scalar extensions by $\bA$ of affine parts of $\cX = \cX^\K$.
More formally, if $\phi:\bB \to \bA$ is a morphism of Weil algebras (scalar extension of $\bB$ by $\bA$), (1) requires that,
for all $(a,y) \in \cD_2^\bB$, the linear part
$(\cU_{\cX^\phi(a)}, \cX^\phi(y))$ of $\cX^\bA$ is nothing but the algebraic scalar extension
$V^\bA = V \otimes_\bB \bA$ of the $\bB$-linear part
$V=(\cU_a,y)$. 

\begin{theorem}
If $(\ul \cX,\ul J)$ is a Jordan geometry over $\K$, then
$\ul \cU_{a,b}$ is, for all pairs  $(a,b) \in \ul \cX^2$, a symmetric space.
If $(\ul \cX,\ul M)$ is an associative geometry over $\K$, then $\ul \cU_{a,b}$ is, for all pairs  $(a,b) \in \ul \cX^2$, 
 a Lie group (with atlas),
and condition (2) is then automatically satisfied.
\end{theorem}

\begin{proof}
We know that, for the Jordan structure map $J^\bA$,
 $\cU_{a,b}^\bA$ is a set theoretic reflection space, and condition (\ref{eqn:-1}) holds by property (2)
of a $\K$--Jordan geometry.
For an associative structure map, $\cU_{ab}^\bA$ is a group, depending functorially on $\bA$,
and having an atlas with single chart $\cU_a$, thus is a Lie group.
As for usual Lie groups,
the tangent map of the inversion map ($J^{ab}_x$, in our case) at the unit element is minus the identity
(cf.\ \cite{Be14}), and hence (2) is automatic.
\end{proof}

\begin{theorem}
Assume $2$ is invertible in $\K$. 
Then Condition (2) follows from the remaining properties of a $\K$-Jordan geometry, and
 any generalized projective geometry (cf.\ Theorem \ref{th:gpg})
gives rise to a Jordan geometry over $\K$.
\end{theorem}

\begin{proof}
If $2$ is invertible in $\K$, and $\cX$ connected, then all inversions $J^{ab}_x$ are conjugate among each other
(see remarks in \ref{rk:midpoints}), so in particular, $J^{ab}_x$ and $J^{xx}_a$ are conjugate. But $J^{xx}_a$ is multiplication
by $-1$ in $(\cU_{a},x)$, and hence its tangent map is minus the identity. The first claim follows
since every geomery can be decomposed into connected components. 

As to the second claim, functoriality of the scalar actions maps $S^\bA$ is part of the very definition of
generalized projective geometries in \cite{Be02}, and now the Jordan structure map $J^\bA$ can be defined
as in  Theorems \ref{th:gpg}. 
\end{proof}

\nin Combined with the existence theorem for generalized projective geometries
(\cite{Be02}, Th.\ 10.1), this implies an existence result for Jordan geometries over rings
in which $2$ is invertible (cf.\ Theorem \ref{th:existence} below).

\begin{theorem}
Let $W$ be a $\K$-module, and let $\cX^\bA = \Gras_\bA (W^\bA)$ be the  associative Grassmannian geometry
of $W^\bA$, with its associative structure map $M^\bA$, and its usual transversality relation $\top^\bA$
and scalar action map $S^\bA$. Then these data define an associative geometry
$(\ul{\cX}, \ul \top, \ul M,\ul S)$, and hence also a Jordan geometry over $\K$.
\end{theorem}

\begin{proof}
This is immediate from the fact that linear algebra in $W^\bA$ is related to the one of $W$ by the usual
algebraic scalar extension functor (and hence is functorial with respect to any scalar extension $\bA$, not
only for Weil algebras). 
\end{proof}

\nin
A {\em special Jordan geometry} is a Jordan subgeometry of an associative geometry, i.e., essentially, of a
Grassmannian. E.g., Lagrangian geometries are of this type.
% Linear algebra of {\em projective quadrics} is also compatible with scalar extensions... secants...

\section{Infinitesimal automorphisms  and linear Jordan pairs}\label{sec:Infaut}

For any Jordan geometry $(\ul \cX,\ul J)$ over $\K$, and any Weil algebra $\bA$, the geometry
$(\cX^\bA,J^\bA)$ will be called the {\em tangent geometry of type $\bA$}.
It is fibered over  $\cX=\cX^\K$, and 
the $\bA$-tangent space at $x \in \cX$ is the fiber of $\pi$ over $x$, denoted by $T^\bA_x\cX$ or $\cX^\bA_x$.
Since the projection is a homomorphism, the fiber over $(x,a)$, for $(x,a) \in \cD_2$, is a subgeometry.
When $\bA = T\K$, we just speak of ``the'' 
tangent bundle$T\cX$  and tangent spaces $T_a \cX$, resp., pair of tangent spaces $(T_a \cX,T_x \cX)$.

\begin{theorem}[Linearity of $T\cX$]\label{th:linearity}
Assume $(\ul \cX,\ul J)$ is a Jordan geometry over $\K$.
Then the tangent bundle $T\cX$ is a {\em linear bundle}, i.e., the tangent spaces $T_a \cX$ carry a canonical
$\K$-module structure. 
This $\K$-module structure 
coincides with its  $\K$-module structure as a submodule of $(\cU_x,0_a)$ in the geometry $T\cX$
(and hence is independent of the choice of $x \in a^\top$).
Moreover, the translation group $\bfT_a$ acts trivially on the tangent space $T_a \cX$:
$$
\forall g \in \bfT_a, \forall u \in T_a \cX : \qquad g(u)=u .
$$
\end{theorem}

\begin{proof}  
For any Weil manifold $\ul M$, the tangent bundle $TM$ is a linear bundle, and moreover
this linear structure on tangent spaces is induced by any chart of $M$ (see \cite{Be14}, Th.\ 6.3),
which means that it is independent of the choice of $x \in a^\top$.

\ssk
Concerning the last assertion,
using that $T_a(J^{xz}_a) = - \id$, we get

$T_a (L^{xz}_a) =T_a (J_a^{xx} J_a^{xz}) = T_a (J_a^{xx}) \circ T_a(J_a^{xz}) =
(-\id_{T_a \cX})^2 = \id_{T_a\cX}$.
\end{proof}

\nin 
Note that the theorem furnishes an interpretation of the split exact sequence
(\ref{eqn:sequence}) in terms of the linear isotropy representation.
Next, recall from \cite{Be14}, Th.\ 8.4  (see also \cite{Be08}, Section 28, and \cite{Lo69} for the case of ordinary tangent bundles) 
that, for any Weil manifold
$\ul M$ and any Weil algebra $\bA$, there is a canonical bijection between
{\em $\bA$-vector fields} (Weil laws that are sections $\ul X:\ul M \to \ul M^\bA$ of the projection 
$\pi:\ul M^\bA \to \ul M$) and
{\em infinitesimal automorphisms} ($\bA$-Weil laws $\ul F :\ul M^\bA \to \ul M^\bA$ covering the identity:
$\pi \circ \ul F = \pi$). 
This bijection is compatible with additional structure (symmetric space, Jordan or associative geometry...).

\begin{definition}
Let $(\ul \cX,\ul J)$ be a Jordan geometry over $\K$.
For each Weil algebra $\bA$, we define
$\bfG^\bA$ to be the group of bijections of $\cX^\bA$ generated by all inversions coming from  the Jordan structure map
$J^\bA$. 
These groups depend functorially on $\bA$, and the corresponding functor will be denoted by
$\ul \bfG: \bA \mapsto \bfG^\bA$. It is a group object in the category of Weil spaces, called
the {\em group of inner automorphisms of $(\ul \cX,\ul J)$}.
\end{definition}

\begin{remark}\label{th:grading}
In general,  $\ul \bfG$ is not a Weil manifold (it  may fail to have an atlas, already for ordinary infinite dimensional real situations,
cf.\ \cite{BeNe04}); but, in terminology introduced in \cite{Be14}, it is a {\em Weil variety}, which is enough for
defining a Lie algebra of $\ul G$, see below. 
\end{remark}

\begin{theorem}
Let $(\ul\cX,\ul J)$ be a Jordan geometry over $\K$ and $\ul\bfG$ its group of inner automorphisms.
Then there is a split  exact sequence of groups
$$
\begin{matrix} 0 & \to & \g & \to & \bfG^{T\K} & \to & \bfG^\K & \to & 1
\end{matrix},
$$
where $\g = \bfG^{T\K} \cap \Infaut(\cX)$ is the subgroup of infinitesimal inner automorphisms, also called
the {\em group of inner derivations of $(\ul \cX,\ul J)$}.
It is abelian, with group law being pointwise addition in tangent spaces, and denoted by $+$, and it is moreover
a $\K$-module, with scalar action given pointwise in tanent spaces.
Let us fix a base point $(o,o') \in \cD_2 = \cD_2^\K$, which is also identified with the corresponding base point
$(0_o,0_{o'}) \in T \cD_2= \cD_2^{T\K}$. Then every element of $\g$ admits a triple decomposition
 (\ref{eqn:triple}), leading to an additive direct sum decomposition of $\g$ as $\K$-module
$$
\g = \g_{-1} \oplus \g_0 \oplus \g_1 ,
$$
where
\begin{align*}
\g_0 & =  \g \cap \bfG_{o,o'} =  \{ \xi  \in \g \mid \xi(0_o)=0_o, \xi(0_{o'})=0_{o'}  \},\cr
\g_{-1} & =\g \cap \bfT_{o'}^{T\K}  = 
 \{ L^{vo}_{o'}  \mid v \in T_o \cX \}, \cr
\g_{1} & = \g \cap \bfT^{T\K}_{o} = \{ L^{wo'}_o \mid w \in T_{o'} \cX \}.
\end{align*}
The group $\K^\times = \{ r_o^{o'} \mid \, r \in \K^\times\}$ acts, via the adjoint representation $r \mapsto T(r_o^{o'})$,
on these spaces diagonally  with eigenvalues $r,1,r\inv$.
\end{theorem}

\begin{proof}
The split exact sequence exists more generally for automorphism groups of Weil spaces equipped with
$n$-ary multiplication maps, see \cite{Be14}, Th.\ 8.2 and 8.6, 
such as Jordan geometries or
symmetric spaces. Moreover, for the Weil algebra $\bA = T\K$, the kernel is always a $\K$-module
 with  fiberwise defined laws  of addition and scalar action (\cite{Be14}, Th.\ 8.6; 
 see also \cite{Lo69} Lemma 4.2, p.\ 52). 
To get a rich supply of infinitesimal automorphisms, we prove:

\begin{lemma} \label{la:infaut}
The following are infinitesimal automorphisms:
\begin{enumerate}
\item
for  $v,w \in T_p \cX$ and $p \top a$, 
the {\em vertical translation}
$L^{vw}_a$,
\item
for $p \top a$,
the {\em Euler field} $(1+\eps)_p^a$.
\end{enumerate}
\end{lemma}

\begin{proof}
For the vertical translation:
$\pi (L^{vw}_a x ) = L^{\pi (v),\pi(w)}_{\pi(a)} \pi( x) = L^{p,p}_a (\pi(x)) = \pi(x)$,
whence $\pi \circ L^{vw}_a = \pi$.
Concerning the Euler field, note first that from functoriality of the scaling  $S$ we get,
 for all scalars $r \in T\K^\times$,
\[
\pi \circ r_x^a = (\pi r)_{\pi x}^{\pi a} \circ \pi
\]
We apply this to the invertible scalar $r = 1 + \eps$ (whose inverse is $1-\eps$):
since $\pi (1+\eps)=1$, 
we get
$\pi\circ (1+\eps)_p^a  = 1_{\pi p}^{\pi a} \circ \pi = \pi$.
\end{proof}

Now let $\xi$ be an infinitesimal automorphism, then $\xi(o)$ lies in the fiber over $o$,
hence is also transversal to $o'$, so we use  (\ref{eqn:triple}) to   decompose
$$
\xi= L_{o'}^{\xi(o),o} \circ D(\xi) \circ L_o^{-\xi(o'),o'}
= L_{o'}^{\xi(o),o} +  D(\xi)  +  L_o^{-\xi(o'),o'} .
$$
By the lemma,
each of the three terms belongs indeed to $\g$, whence the decomposition.
To prove the claim on the eigenvalues of the scalar action, recall that the
group $\Aut_\K(\cX)$ acts on $\Infaut(\cX)$ by conjugation  (``adjoint representation'')
\begin{equation}\label{eqn:Adj}
\Aut_\K(\cX) \times \Infaut(\cX)\to \Infaut(\cX),\qquad (g,\xi) \mapsto g.\xi:= Tg \circ \xi \circ Tg\inv .
\end{equation}
and now   read equations
(\ref{eqn:ST1}) and (\ref{eqn:ST2}) for infinitesimal arguments $\eps v$ and $\eps a$.
\end{proof} 

\begin{theorem}
With notation from the preceding theorem, the $\K$-module $\g$ carries the structure of a 
$\K$-Lie algebra, with respect to the usual Lie bracket of vector fields (defined, for general Weil
manifolds, in terms of the group commutator in $\bfG^{TT\K}$), and, for any choice of base point
$(o,o') \in \cD_2$, the additive decomposition
$\g = \g_{-1} \oplus \g_0 \oplus \g_1$ from the preceding theorem is a {\em $3$-grading of the Lie algebra $\g$},
that is, it satisfies the bracket rules $[\g_i,\g_j] \subset \g_{i+j}$:
$$
[\g_1,\g_1]=0=[\g_1,\g_{-1}], \quad [\g_1,\g_0]\subset \g_1, \quad [\g_{-1},\g_0]\subset \g_{-1}, \quad 
[\g_1,\g_{-1}]\subset \g_0.
$$
\end{theorem}

\begin{proof}
First of all, let us recall that the Lie bracket of vector fields (infinitesimal automorphisms) is defined
via the group structure of $\bfG^{TT\K}$, where
$$
TT\K = \K[\eps_1,\eps_2] = \K \oplus \eps_1 \K \oplus \eps_2 \K \oplus \eps_1 \eps_2 \K
\quad (\eps_1^2 = 0 = \eps_2^2)
$$
via the group commutator (see \cite{Be14}, cf.\ also \cite{Be08, KMS}):
\begin{equation}\label{eqn:Liebracket}
\eps_1 \eps_2 [X,Y] = \eps_1 X \cdot \eps_2 Y \cdot (\eps_1 X)\inv \cdot (\eps_2 Y)\inv .
\end{equation}
It is then a general fact that the group of (inner) derivations of some algebraic structure (symmetric spaces,
Jordan geometries) is stable under the Lie bracket (\cite{Be14}, Cor.\ 8.7), hence $\g$ is a $\K$-Lie algebra.
With respect to a base point $(o,o')$, $\g_1$ and $\g_{-1}$ are just the Lie algebras of the translation groups
$\bfT = \bfT_{o'}$ and $\bfT' = \bfT_o$, and hence are abelian subalgebras.
Moreover, $\bfT$ is normal in $\bfG_{o'}$ (lemma  \ref{la:normal}); the same is then true for
$\bfT^{TT\K}$ in $\bfG_{o'}^{TT\K}$, which implies that $[\g_1,\g_0]\subset \g_1$, and similarly we get
$[\g_{-1},\g_0]\subset \g_{-1}$.
It remains to prove that $[\g_1,\g_{-1}]\subset \g_0$.

\begin{lemma}
The Euler operator $E$ (Lemma \ref{la:infaut}) acts via $\ad(E)$ with eigenvalues $i$ on $\g_i$, $i=1,0,-1$,
and $\g_0$ is equal to its $0$-eigenspace.
\end{lemma}

\begin{proof} 
To simplify notation, we identify $\g_1$ with $V = T_o \cX$ and $\g_{-1}$ with $V'=T_{o'} \cX$.
Then
$\ad(E)$ commutes with all elements $H \in \g_0 = \g \cap \bfG_{o,o'}$ because $H$ acts $TT\K$-linearly 
on $TTV \times TTV'$, hence commutes with the scalars $1+\eps_1$ or $1+\eps_2$.
In order to compute $\ad(E)v$ via (\ref{eqn:Liebracket}),
we compute the commutator with elements from $\g_1$:
$$
(1+\eps_1) L_{o'}^{\eps_2 v,o} (1-\eps_2) L_{o'}^{-\eps_2 v,o} =
L_{o'}^{(1+\eps_1) \eps_2 v,o} L_{o'}^{-\eps_2 v,o} = 
L_{o'}^{\eps_1 \eps_2 v,o},
$$
implying that $[E, v] = v$ for all $v \in V = \g_1$.
Since $\ad(E)$ acts by $1-\eps$ on $V'$, the same computation yields that
$[E,w]=-w$ for all $w \in V'=\g_{-1}$.

Conversely, if $\ad(E)X=0$, then decompose  $X = X_1 + X_0 + X_{-1}$ with $X_i \in \g_i$, 
to get
$ X = X - \ad(E)^2 X = X - X_1 - X_{-1} = X_0$, whence $X \in \g_0$.
\end{proof}

Now, let $X = [Y,Z] \in [\g_1,\g_{-1}]$ with $Y\in \g_{1}$, $Z \in \g_1$.
Then
 the Jacobi identity implies that $\ad(E)X=[E,[Y,Z]] = [[E,Y],Z]+[Y,[E,Z]] = - [Y,Z]+[Y,Z]=0$,
hence $X \in \g_0$ by the lemma.
\end{proof}

As is well-known (cf.\, e.g., \cite{Be00}),
for every $3$-graded Lie algebra $\g = \g_1 \oplus \g_0 \oplus \g_{-1}$, the pair of $\K$-modules
$V^\pm := \g_{\pm 1}$ becomes a {\em linear Jordan pair} with trilinear maps
\begin{equation}
V^\pm \times V^\mp \times V^\pm \to V^\pm, \qquad
(x,a,z) \mapsto \{ xaz\} :=  [[x,a],z] \, ,
\end{equation}
i.e., the trilinear maps satisfy  the identities

\msk
(1) $\{ xaz \} = \{ zax \}$

(2) $\{ uv \{ xyz \} \} = \{ \{ uvx \} yz \} - \{ x \{ vuy \} z \} + \{ xy \{ uvz \} \}$ 
%= JP14 from \cite{Lo75}

\msk \nin
Combining this with the preceding theorem gives immediatly:

\begin{theorem}[The linear Jordan pair of a Jordan geometry]\label{th:LJP}
Assume $\ul \cX$ is a Jordan geometry over $\K$ with base point $(o,o')\in \cD_2$.
Then the pair of $\K$-modules $(V^+,V^-)=(\g_1,\g_{-1})$ becomes a linear Jordan pair
with respect to
$\{ x a z \} = [[x,z]a]$.
% where $Q$ is   the quadratic map defined by (\ref{eqn:Q-def2}).
The Jordan pair depends functorially on the geometry with base point.
\end{theorem}

\begin{proof}
Only the last assertion remains to be proved.
Indeed, the triple Lie bracket can be interpreted as the Lie triple system belonging to the (polarized)
symmetric space $\cD_2$ (cf.\ Theorem \ref{th:PSS}; see also \cite{Be00}), and the Lie triple system of a symmetric
space depends functorially on the space with base point (\cite{Be08, Be14}).
\end{proof}

\section{Quadratic vector fields and quadratic Jordan pairs}\label{sec:JGtoJP}

\subsection{The quadratic map}
The link of the Jordan pair with the geometry of $(\cX,J)$ becomes more direct if 
we look at {\em quadratic Jordan pairs} instead of linear ones. 
In a first step, we realize the Lie algebra $\g$ is a space of {\em quadratic vector fields on $V$},
where, as above $V = \cU_{o'} \cong \bfT_{o'}$.
If $\xi$ as an infinitesimal automorphism corresponding to a vector field $X$,
we call $X=X_\xi$ also
the {\em associated vector field} of $\xi$. Since $\xi$ acts by translations in each
fiber, and
writing $TV = V \oplus \eps V$, the map  $\xi$ has the chart representation
\begin{equation}\label{eqn:vf-chart}
\xi(x + \eps v) = x + \eps (v+ X(x)) .
\end{equation}
%where $X\vert_V :V \to \eps V \cong V$  is the chart representation of a vector field.

\begin{theorem}\label{th:quadratic}
For all $\xi \in \g(\cX)$, the vector field $X\vert_V:V \to V$ representing $\xi$, is
a quadratic polynomial. More precisely, this polynomial is
constant if $\xi \in \g_{-1}$, linear if $\xi \in \g_0$, and homogeneous of degree $2$ if
$\xi \in \g_1$. Thus $\g$ is represented over $V$ by the Lie algebra of quadratic polynomials
\[
\bigl\{   X : V \to V \mid \, 
X(x) = v + Hx + Q(x)a , \quad  v \in V, H \in \g_o, a \in V' \bigr\} ,
\]
where the polynomial
\begin{equation}
Q : V \times V' \to V, \quad (x,a) \mapsto Q(x)a,
\end{equation}
quadratic in $x$ and linear in $a$, is defined by
\begin{equation}\label{eqn:Q-def}
L_o^{ \eps a, o'}  (x+\eps v ) = x + \eps (v  +  Q(x) a) .
\end{equation}
Similarly, $\g$ is also represented by quadratic polynomial vector fields over $V'$.
\end{theorem}

\begin{proof}
If $\xi \in \g_{-1}$, then $\xi = L^{\eps v,o}_{o'}$ for some $v \in V$, hence
$\xi(x)=x+\eps v$, and the corresponding
vector field is $X(x)= v$, which is a constant function.
If $\xi \in \g_0$, then $\xi$ acts linearly on $V$ (and on $V'$).
It remains to show that $X$ is homogeneous quadratic polynomial if $\xi \in \g_1$.

\ssk
In the chart formula, the adjoint action (\ref{eqn:Adj})  is described as follows: 
using (\ref{eqn:vf-chart}) together with
$Tg(x+\eps v)=g(x) + \eps df(x)v$, we get, whenever $g\inv .x \in V$, 
\begin{equation}\label{eqn:V-action}
(g.\xi)  (x) = dg(g\inv .x)  \cdot X(g\inv.x) .
\end{equation}
If $g=L^{v,o}_{o'}$ is a translation with $v\in V$, (\ref{eqn:V-action})  gives
\begin{equation}\label{eqn:V-action2}
(L^{v,o}_{o'} .\xi) (x) = X(x-v),
\end{equation}
and for a major dilation $g=r_o^{o'}$ it gives
\begin{equation}\label{eqn:V-action3}
(r_o^{o'} .\xi) (x) = r X(r\inv x) \, .
\end{equation}
Specializing (\ref{eqn:V-action3}) to vector fields $X$ from the three parts $\g_i$, 
$i=-1,0,1$, we get that $X:V \to V$ is homogeneous of degree
$0$, $1$ or $2$, respectively, since $\xi$ is eigenvector  for the
$\K^\times$-action for the eigenvalues $r,1,r\inv$, respectively, by Theorem \ref{th:grading}.

\ssk
Now let $\xi = L_{o}^{o',\eps a} \in \g_1$ ($a \in V'$) and $X$ its associated vector field. 
In order to show that $X(x)$  is quadratic polynomial, it remains to show that
the map 
$$
X_v:V \to V, \qquad 
x \mapsto X_v(x) = X(x-v)-X(x),
$$ 
is affine, for all $v \in V$. (Equivalently, that $x \mapsto X(x+v)-X(x)-X(v)$ is linear.)
According to (\ref{eqn:V-action2}), the field 
$X(x-v)$ represents 
the infinitesimal automorphism
$Tg \circ \xi \circ Tg\inv$ where $g=L_{o'}^{v,o}$ is translation by $v$,
and $-X(x)$ represents $\xi\inv$, whence $X_v$ 
represents the infinitesimal automorphism
$Tg \circ \xi \circ Tg\inv \circ \xi\inv$,  that is, 
it represents
$$
\xi_v:= 
L_{o'}^{v,o} L_{o}^{o',\eps a} L_{o'}^{o,v} - L_{o}^{o',\eps a}
=
 L_{o'}^{v,o} L_{o}^{o',\eps a} L_{o'}^{o,v} L_o^{\eps a,o'} .
$$
Saying that $X_v$ is affine amounts to saying that $\xi_v$ fixes $o'$. 
Now, 
$$
\xi_v(o') = 
(L_{o'}^{v,o} L_{o}^{o',\eps a} L_{o'}^{o,v} - L_{o}^{o',\eps a})o' =
L_{o'}^{v,o} L_{o}^{o',\eps a}  - (-\eps a) =
L_{o'}^{v,o} (-\eps a)  - (-\eps a)
$$
But $L_{o'}^{v,o} (-\eps a) = T_{o'} (L^{v,o}_{o'}) (-\eps a)$
is the tangent map of $L_{o'}^{v,o}$ at its fixed point $o'$, applied to $-\eps a$.
According to Theorem  \ref{th:linearity}, this tangent map is the identity, and hence 
it follows that $\xi_v(o') = o'$, hence $\xi_v$ is affine for all $v\in V$ and thus $\xi$ is quadratic.
The map $Q(x)a$ is defined by $Q(x)a=X(x)$ for $\xi$ as above, and hence is homogeneous
quadratic in $x$. It is linear in $a$ since the map $a \mapsto L^{\eps a,o'}_o$ is a linear
isomorphism from $V'$ to $\g_1$. 
\end{proof}

\begin{definition}
With respect to a fixed origin $(o,o')$ and model space $(V,V')$, we define the {\em quadratic map}
as above via ``quasi-translation by $\eps a$''
\begin{equation}\label{eqn:Q-def1}
Q:V \to \Hom(V',V),\quad x \mapsto (a \mapsto Q(x)a = L^{o',\eps a} (x) = L^{-\eps a,o'}(x)) ,
\end{equation}
and we define a map that is bilinear symmetric in $(x,v)$ and linear in $a$,
\begin{equation}\label{eqn:Q-def2}
Q(x,v)a:= D(x,a)v:= Q(x+v)a-Q(x)a -Q(v)a .
\end{equation}
Maps $Q':V' \times V' \to \Hom(V,V')$ and $D': V' \times V \to \End(V')$ are defined dually.
\end{definition}

\subsection{The quadratic Jordan pair}

\begin{definition}\label{subsec:QJP}
A {\em quadratic Jordan pair} is a pair $(V^+,V^-)$ of $\K$-modules 
together with quadratic maps $Q_\pm : V^\pm \to \Hom(V^\mp,V^\mp)$ such that the following identities
hold in all scalar extensions (see \cite{Lo75};  superscripts $\pm$ are omitted)\footnote{As Loos remarks in loc.\ cit., p.1.3, it suffices to consider 
scalar extensions by $\K[X]/(X^k)$ for $k=2,3$; in particular, it suffices to consider 
Weil algebras.} 

\msk
(JP1) $D(x,y) Q(x) = Q(x) D(y,x)$

(JP2) $D(Q(x)y,y) = D(x,Q(y)x)$

(JP3)  $Q(Q(x)y) = Q(x) Q(y) Q(x)$, where

\msk
\nin
$\{ xyz \} := D(x,y)z:= Q(x,z)y :=Q(x+z)y-Q(x)y-Q(z)y$,
so $\{xyx \} = 2 Q(x)y$.
\end{definition}

\nin It is shown in \cite{Lo75}  that every quadratic Jordan pair is linear, and that the converse
is true if $V$ has no $6$-torsion. 

\begin{theorem}[The quadratic Jordan pair of a Jordan geometry]
Assume $(\ul \cX,\ul J)$ is a Jordan geometry over $\K$ with base point $(o,o')\in \cD_2$.
Then the pair of $\K$-modules $(V^+,V^-)=(\cU_{o'},\cU_o)$ becomes a quadratic Jordan pair
with respect to the maps $Q_+=Q$ and $Q_- = Q'$ defined  by
(\ref{eqn:Q-def1}).
The quadratic map, and hence the Jordan pair, depend functorially on the geometry with base point.
\end{theorem}

\begin{proof}
If $V$ has no $6$-torsion, by the preceding remarks, the Jordan pair is linear, and hence
 the claim follows from Theorem \ref{th:LJP}.
In the general case, one can adapt to our framework the arguments given in the proof of \cite{Lo79}, Th.\ 4.1;
however, since the computations are fairly long and involved, we will not reproduce them here in full
detail. 
The main ingredients used in loc.\ cit.\ are the  relations between ``usual'' translations and
quasi-translations (in our framework: Lemma \ref{la:Berg}), and 
the behavior  of (quasi-) translations with respect to scalars ((\ref{eqn:ST1}), (\ref{eqn:ST2}));
these relations furnish a description of the elementary projective group by generators and relations, from
which the Jordan pair identities are deduced by using algebraic differential calculus in the setting of
algebraic geometry  (\cite{Lo79}, page 40). 
All of these arguments carry over to our setting; we only have to replace
 the argument of Zariski-density used  repeatedly in loc.\ cit.\
%(and which applies only in the finite dimensional case) 
by the following more general argument, which in turn is a geometric version of ``Koecher's principle
on identitities'' saying that {\em a Jordan polynomial which vanishes in all quasi-invertible Jordan pairs
is zero} (see  \cite{Lo95}, p.\  97, for this formulation).

\begin{lemma}[``Koecher's principle'']
\label{la:density}
A Jordan polynomial which vanishes on all quasi-invertible quadruples of Jordan geometries is zero.
More formally, this means:  assume 
$P = P_\cX$ is a Weil   law, depending functorially on Jordan geometries $\cX$, such that $P_\cX$ is
 defined for quadruples $(a,o,o',x) \in \cD_4(\cX)$ and 
is polynomial in $(a,x)$ for all fixed base points $(o,o') \in \cD_2(\cX)$;  
if $P$ vanishes for all quasi-invertible quadruples
$(a,o,o',x) \in \cD_4'(\cX)$ in all Jordan geometries $\cX$, then $P=0$.
\end{lemma}

\nin {\em Proof of the lemma.} 
Considering $(o,o')$ as fixed, we suppress it in the notation.
Let $P(a,x)=0$ be a polynomial identity of degree $k$, valid for all quasi-inverible quadruples
$(a,o,o',x) \in \cD_4'$ in Jordan geometries $\cX$.
Let $J^k \cX$ be the scalar extension of $\cX$ by the {\em jet ring}
$J^k\K:= \K[X]/(X^{k+1}) = \K[\delta]$ (see (\ref{eqn:jetring})).
Just as in case $k=1$ (tangent bundle), $J^k \cX$ is a bundle over $\cX$, called the
{\em $k$-th order jet bundle of $\cX$}: in every chart of the canonical atlas, it has a product structure,
and likewise, the set $\cD_4' (J^k \cX)$ is a bundle over $\cD_4'(\cX)$.
Fixing $(o,o')$ as base point, all elements
$(\delta a, \delta x)$ with $(x,a) \in V^+ \times V^-$ are hence quasi-invertible
(i.e., $(\delta a,o,o',\delta x)$, lying in the fiber over $(o',o,o',o)$, belongs to $\cD_4'(J^k \cX)$).
By assumption, we thus have $P(\delta x,\delta a)=0$.
Expanding this polynomial and ordering according to powers
$\delta,\delta^2,\ldots,\delta^k$, we see  that all homogeoneous parts of the polynomial $P$ vanish,
and hence $P=0$. This proves the lemma and the theorem. 
% example:
%for instance, for a quadratic identity $P(x)=0$, valid for $x \in U$,
%extend with $n=3$, eg. for $P(x)=x^2$,
%$$
%(x+\delta v)^2 = x^2 + 2 \delta x + \delta^2 v^2,
%$$
%and comparing coefficients at $\delta^2$ shows that the initial identity $P(v)=0$ holds indeed for all
%$v \in V$, since all $x+\delta v$ belong to $TU$...
\end{proof}

\nin
The proof of the lemma takes up the  idea that, geometrically,  ``modules'' of Jordan pairs should
correspond to {\em bundles in the category of Jordan geometries}. 
Vector bundles then correspond to {\em representations} in the sense of \cite{Lo75}, 2.3; they are scalar
extensions by Weil algebras $\bA =\K\oplus \mbA$ such that $\mbA$ has zero product
(vector algebras, cf.\ \cite{Be14}).  In this context, the proof of the lemma
leads to a geometric version of the {\em permanence principle}
\cite{Lo75}, 2.8.

\begin{theorem}[The Jordan triple system of a Jordan geometry with polarity]
Let $(\ul \cX,\ul J)$ be a Jordan geometry over $\K$ with polarity $\ul p:\ul \cX \to \ul \cX$
and base point $(o,o')\in \cD_2$ such that $o' = p(o)$.
Then the $\K$-module  $V = U_{o'}$ becomes a quadratic Jordan triple system
with respect to the map  $Q(x)y = p Q^\pm(x) p(y)$.
\end{theorem}

\begin{proof}
The polarity $p$ defines an involution of the Jordan pair $(V^+,V^-)$ from the preceding theorem,
and a Jordan pair with involution is the same as a Jordan triple system (cf.\ \cite{Lo75}).
\end{proof}

\section{Jordan theoretic formulae for the inversions}\label{sec:Fomulae}

Having  the Jordan pair $(V^+,V^-)$ associated to $(\ul \cX,\ul J)$ at our disposition, we wish to describe
the geometric structure of $(\cX,J)$ in terms of the Jordan pair,
%%%%%%%%%%%%%%%%%%%%%%%%%%%%%%%%%%
by  giving  Jordan theoretic formulae expressing $J^{xz}_a(y)$ and
$J^{xz}_a(b)$ in terms of the  Jordan pair. Notation for Jordan pairs is as in definition \ref{subsec:QJP};
in order to simplify formulas, we suppress subscripts $\pm$, by assuming always
that $o,v,x,y,z \in V^\pm$ and $o',a,b,c \in V^\mp$. 

\begin{definition}
As usual in Jordan theory  (cf.\ \cite{Lo75}), one defines:

\begin{enumerate}
\item the {\em
Bergman operator} is defined by
$B(y,b)x:= x - D(y,b) x + Q(y) Q(b) x$, 
\item
$(x,a)$ is called {\em quasi-invertible} if $B(x,a)$ is invertible, and then one defines 
\[
x^a := B(x,a)\inv \bigl(x-Q(x)a \bigr) \, ,
\]
and then
 the {\em inner automorphism defined by $(x,b)$} is given by
\[
\beta(x,a) := \bigl( B(x,a), B(a,x)\inv \bigr) .
\]
\end{enumerate}
\end{definition}

\nin 
 In order to obtain the general formulae for $J^{xz}_a$, we proceed in three steps (the reader may compare
 with the example of the projective line given in the introduction):
 
 \ssk
 Step 1:  two of the points are base points (case $(z,a)=(o,o')$, Lemma \ref{step1})
 
 Step 2: one point among $x,z$ is the base point $o$ (Lemma \ref{step2})
 
 Step 3: general case -- all three points different from base points (Theorem \ref{step3}).
 
 \begin{lemma}\label{step1}
 For $x,v \in V^+ = \cU_{o'}$ and $a  \in V^- = \cU_o$, the following holds:
 the affine space structure of $V^+$ (and dually of $V^-$) is described by the
translations: $L_{o'}^{vo}(x)=v+x$ and the
major dilations: $r_x^{o'}(y) = (1-r) x + ry$, and in particular,
$J^{oo}_{o'}(x)=-x=(-1)_o^{o'}(x)$.
The Jordan theoretic
Bergman operator coincides with the ``geometric Bergman operator'' defined in (\ref{eqn:B}) :
$\beta(x,a) = B^{o,o'}_{xa}$, and
$x \top a$ if, and only if, $(x,a)$ is quasi-invertible, and then
\[
L_o^{ao'}(x)  = x^a = B(x,a)\inv \bigl(x - Q(x)a \bigr) .
\]
If $(-v,a)$ is quasi-invertible, then
\[
J_o^{ao'}(v)  = L_o^{ao'} (-v) = (-v)^a =  - B(-v,a)\inv \bigl(Q(v)a +v \bigr) ,
\]
\end{lemma}

\begin{proof}
The statement on the action of $L_o^{a,o}$ and $J_o^{ao'}$ by (quasi-)inverses is proved
in \cite{Lo79}, Lemma 4.7 (the framework in loc.\ cit.\ is slightly different from ours, but the proof carries over
by using Lemma \ref{la:density}; see also
\cite{BeNe04}, Section 3 for another proof in a different setting). 
\end{proof}

\begin{lemma}\label{step2}
If $(-v,a) \in V^+ \times V^-$ is quasi-invertible, then
\[
J_o^{ao'}(v)  = L_o^{ao'} (-v) = (-v)^a =  - B(-v,a)\inv \bigl(Q(v)a +v \bigr) ,
\]
and, with $v' := J_a^{vo}(o') = 2a + Q(a)v$, we have the triple decomposition (\ref{eqn:triple})
\[
J_a^{vo} = L_{o'}^{v,o} \circ (- \beta(v^{-a},a)) \circ L_o^{o',v'} =  L_{o'}^{v,o} \circ (- \beta(-v,a))\inv  \circ L_o^{o',v'} \, .
\]
\end{lemma}

\begin{proof}
To compute $J^{vo}_{a}$, note that
\[
J^{v,o}_a = J_o^{ao'} J_{o'}^{J_o^{ao'} v,o} J_o^{ao'} = J_o^{ao'} J_{o'}^{w,o} J_o^{ao'} 
\]
with $w= J_o^{ao'} v=(-v)^a$.
Now
use the ``commutation relation'' from Lemma \ref{la:Berg}  
\[
J_o^{ao'} J_{o'}^{wo} =
J_{o'}^{J_o^{o'a}(w),o}   \, \beta(w,-a)  \,  J_o^{o',J_{o'}^{ow}(a)}
\]
to get the triple decomposition
\begin{align*}
J^{v,o}_a =  J_{o}^{ao'} J_{o'}^{w,o}   J_{o}^{ao'}  & =
J_{o'}^{J_o^{o'a}(w),o}   \, \beta(w,-a)  \,  J_o^{o',J_{o'}^{ow}(a)}  J_{o}^{ao'} 
\cr
& = L_{o'}^{J_o^{o'a}(w),o}   \, \bigl( -\beta(w,-a) \bigr)   \,  L_o^{J_{o'}^{ow}(a),a} 
\cr
& =L_{o'}^{v,o}   \, \bigl( -\beta((-v)^a ,-a) \bigr)   \,  L_o^{o',v'}
\cr
& =L_{o'}^{v,o}   \, \bigl( -\beta(-v ,a) \bigr)\inv   \,  L_o^{o',v'}
\end{align*}
where $v' = a - J_{o'}^{ow}(a)$; note also that
$\beta((-v)^a,-a)) = \beta(v^{-a},a)=\beta(v,-a)\inv$ (identity JP 35 from \cite{Lo75}).
We give another expression for $v'$ by using the {\em symmetry principle}
$x^y = x + Q(x)y^x$ (\cite{Lo75}, Prop.\ 3.3)  
\begin{align*}
v' = a - J_{o'}^{ow}(a) & = a - (-a)^w = a + a^{-w}
\cr
& = 2a + Q(a) (-w)^a = 2a + Q(a) v   \, .
\end{align*}
This proves the triple decomposition for $J^{vo}_a$ given in the claim.
\end{proof}

\nin
By uniqueness of the triple decomposition, it follows that
\begin{equation}
J^{v,o}_a(o')= 2a + Q(a) v .
\end{equation}
The formula for $J^{vo}_a$ has also been given, in another framework, in \cite{Be08}, Th.\ 2.2. 

\begin{theorem}\label{step3}
For $x,y,z \in V^+$ and $a,b,c \in V^-$, the following holds:
if $(x,-a)$ and $(z,Q(a)z)$ are quasi-invertible, then
we have the triple decomposition (\ref{eqn:triple})
$$
J^{xz}_a = L^{vo}_{o'} \circ h \circ L^{o' v'}_o ,
$$
where
\begin{align*}
v = J^{xz}_a (o) & = (xoz)_{a} = (x^{-a} + z^{-a})^a 
= x + B(x,-a) z^{Q(a) x} \, ,
\cr
v' = J^{xz}_a(o') & =   2a +  Q(a)x  +  Q(a) B(x,-a) z^{(Q(a)x)} 
\cr &= 2a +  Q(a)x +  B(a,-x) (Q(a)z))^x, 
\cr
h = D(J^{xz}_a) & = - \beta\bigl((-v)^a,-a \bigr) = - \beta\bigl( (-x)^a + (-z)^a, -a \bigr) \, .
% =?  check::  \beta(x,-a) \beta(z,Q(a)x)\inv \beta(z,-a)  \, .
\end{align*}
Using this notation, 
 the action of $J^{xz}_a$ on $V^+$, resp.\ on $V^-$, is given by
\begin{align*}
J^{xz}_a(y) &= v - \beta(v^{-a},a) y^{-v'} 
\cr
& = x + B(x,-a) z^{Q(a) x} \,  - B \bigl( x^{-a} + z^{-a},- a \bigr) y^{(- 2a -  Q(a)x  -  B(a,-x) (Q(a)z)^x)}
\cr
J^{xz}_a(b) & = v' - \beta(v^{-a},a)\inv b^{-v} = v' - \beta(v,a) b^{-v} =
v' - B(a,v)\inv b^{-v} 
\cr
& =
 2a +  Q(a)x +  B(a,-x) (Q(a)z))^x -   \cr 
& \qquad \qquad  B\bigl(a, x+B(x,-a)z^{Q(a)z}  \bigr) \cdot  b^{(-  2a -  Q(a)x -  B(a,-x) (Q(a)z))^x )}
\end{align*}
If, moreover, $(y,-a)$ is quasi-invertible, we have also
\[
J^{xz}_a(y) = (x^{-a} - y^{-a}+ z^{-a})^a .
\]
\end{theorem}

\begin{proof}
Using the ``transplantation formula
(\ref{eqn:Transp}),
$J^{xz}_a = J^{J^{xz}_a(o),o}_a = J_a^{v,o}$, 
we have to compute the value $v = J^{xz}_a(o)$, and then apply the preceding theorem to get
the expressions from the claim. 
To compute $J^{xz}_a(o)$, start by observing that 
$$
J_{o'}^{xz}(y)= (xyz)_{o'} = x -y +z 
$$
whence, using that $L_{o}^{ao'} (o')=a$ and, by the preceding theorem, $L_o^{o'a}(y)= y^{-a}$,
$$
J_a^{xz} (y) = J_{L_o^{ao'} (o')}^{xz}(y)=
L_o^{ao'} J_{o'}^{L_o^{o'a}x,L_o^{o'a}z} L_o^{o'a}(y) =
(x^{-a} - y^{-a}+ z^{-a})^a \, ,
$$
proving the last formula from the claim, which for $y=o$ gives
$J^{xz}_a(o)=(x^{-a}+z^{-a})^a$.
By using \cite{Lo75}, Th.\ 3.7: $(x+z)^y=x^y+B(x,y)\inv .z^{(y^x)}$ and $x^{y+z}=(x^y)^z$, as well as
(JP35) $B(x,y)\inv= B(x^y,-y)$ and the ``symmetry principle'' 
$x^y = x + Q(x)y^x$,  we get the following Jordan theoretic formula
\begin{align*}
v= (x^{-a}+z^{-a})^a & = 
(x^{-a})^a + B(x^{-a},a)\inv (z^{-a})^{(a^{(x^{-a})})}
\cr
& =  x + B(x,-a) z^{(a^{(x^{-a})} -a)} 
\cr
& = x + B(x,-a) z^{(Q(a)x)} \, .
\end{align*}
It follows that
\begin{align*}
 J^{xz}_a(o') =  v'   & =  2a + Q(a) v   \cr
& = 2a + Q(a) x + Q(a) B(x,-a) z^{(Q(a)x)} \, .
\end{align*}
Now replace $v$ and $v'$ by these expressions in the triple decomposition from the preceding theorem.
For $J^{xz}_a(y)$, the result drops out immediately; for $J^{xz}_a(b)$,  one uses first that 
$J^{xz}_a = (J^{xz}_a)\inv$.
\end{proof}

Note that
there are other ways to compute the values of $J^{xz}_a(y)$ and of $J^{xz}_a(b)$, and equality of the results
then often  corresponds to certain Jordan-theoretic identities. 

%%%%%%%%%%%%%%%%%%%%%%%%%%%%%%%%%%%

\section{Unital Jordan and associative algebras}\label{sec:JA}

{\em Unit elements} in algebras (Jordan or associative) come from {\em closed transversal triples}:
assume $(a,b,c)=(o,o',e)$ is a pairwise transversal triple in  a Jordan geometry $(\ul \cX,\ul J)$.
According to Lemma \ref{la:ssa}, the set $U=U_{oo'}$ is a symmetric space with product
$s_x(y)=J^{oo'}_x(y)$. Since $J^{oo'}_x$ exchanges $o$ and $o'$, it induces a $\Z$-linear
bijection of $V = U_{o'}$ onto $V' = U_o$.  Fix the  point $e$
as base point  in $U$, and define, for $x \in U$, a linear map
\begin{equation}\label{eqn:JAQ}
Q_x := Q_{xe}^{oo'} = J_x^{oo'} \circ J_e^{oo'}\vert_V  : V \to V .
\end{equation}
Since $Q_x (Q_y)\inv y = J_x^{oo'} J_e^{oo'} J_e^{oo'} J_y^{oo'} y = J_x^{oo'} (y) = s_x(y)$,
the structure of $U$ can be entirely described in terms of the map 
$U \times V \to V$, $(x,y) \mapsto Q_x(y)$.

\begin{theorem}[The Jordan algebra of a Jordan geometry with pairwise transversal triple]
Let  $(\ul \cX,\ul J)$ be a Jordan geometry over $\K$ with pairwise transversal triple $(a,b,c)$.
Choose $(a,b)=:(o,o')\in \cD_2$ as base point.
Then the  $\K$-module  $V = U_{o'}$ becomes a quadratic Jordan algebra with quadratic map
$U_x(y) = Q(x) Q(e)\inv y$ and
with unit element $e=c$.  
The  set  $V^\times$ of invertible elements agrees with the symmetric space $U = V \cap V'$,
and the quadratic map $Q(x)$ agrees with $Q_x$ defined by (\ref{eqn:JAQ}).
\end{theorem}

\begin{proof}
We have to show that the Jordan pair $(V,V')$ associated to the base point $(o,o')$ has invertible elements
(cf.\ \cite{Lo75}); more precisely, we show that every element $x$ from $U= U_{oo'}$ is invertible.
Indeed, this follows from the fact that $j:= J^{oo'}_e$ is an automorphism of $\g$ exchanging $o$ and $o'$,
hence exchanging also
$\g_1$ and $\g_{-1}$: using numerators and denominators, it is shown exactly as in \cite{BeNe04}, Section 5.1, that,
for all $x \in V^\times$, we have the formula
\[
j(y) = Q(e) Q(y)\inv Q(e) y .
\]
In particular, since $j(e)=e$, it follows that $e$ is an invertible element, thus
$(V,e)$ is a quadratic Jordan algebra with Jordan inversion $j$ (\cite{Lo75}).
From this it follows is in \cite{BeNe04} that $Q(x)=Q_x$ and that $V^\times = U$.
\end{proof}

\begin{theorem}[The associative algebra of an associative geometry with transversal triple]\label{th:AA}
Assume $(o,o',e)$ is a closed transversal triple in an associative geometry
$(\ul \cX,\ul M)$. Then
the group law of $U_{oo'}$ extends to an associative algebra structure on $V= U_{o'}$, with
bilinear product induced by the second tangent law $TTU_{oo'}$.
\end{theorem}

\begin{proof}
Let $V:=V_{o'}$, $V' := V_o$ and $V^\times := U_{oo'} = V \cap V'$.
From the properties of an associative geometry, it follows that $(V^\times ,e)$ is a Lie group
with group law
\[
xz = (xez)_{oo'} =  M_{xz}^{oo'}(e)=L_{xe}^{oo'}(z) = R_{ez}^{oo'}(x) ,
\]
which is bilinear for the linear structure $(V,o)$.
Let $m:U\times U \to U$ be the group law of the Lie group $U=V^\times$;
then the group law of $TTU$ is given by $TTm$ which is scalar extension of $m$ by the ring
$TT\K=\K[\eps_1,\eps_2]$.  The map
\[
\eps_1 V \times \eps_2 V \to \eps_1 \eps_2 V, \qquad
(\eps_1 u,\eps_2 v) \mapsto (\eps_1u)(\eps_2 v) = TTm(\eps_1 u,\eps v)
\]
is bilinear, since, for one of the arguments fixed, the remaining map is a tangent map.
Thus a bilinear product $uv$ on $V$ is defined by requiring
\[
\eps_1 \eps_2 (uv) := (\eps_1 u) (\eps_2 v) .
\]
The group law $T^3 m$ on $T^3 U$ is associative, thus, in particular,
$\eps_1 u (\eps_2 v \cdot \eps_3 w) = (\eps_1 u \cdot \eps_1 v) \eps_3 w$,
which, by definition of the product, yields
$u(vw)=(uv)w$.
Thus $V$ with product $uv$ is an associative algebra.
Moreover, if $u,v \in U$, then left and right multiplications $L_u:V \to V$
and $R_v:V \to V$, are linear maps, hence agree with their tangent maps, implying that
the products $uv$ taken in $U$ and in $V$ agree.
\end{proof}

\begin{remark}
If the geometry is not self-dual, then, for a fixed base point $(o,o')\in \cD_2$, the
pair $(V,V')$ becomes an {\em associative pair} (see \cite{BeKi1} for relevant
definitions). 
%The definition of the trilinear products
%is similar as above, but technically more complicated.
% -- we will not go here into details,
%but just mention that the approach from loc.\ cit.\ is different than the one chosen here:
%instead of using algebraic differential calculus, an extension of the torsor law
%$(x,y,z) \mapsto (xyz)_{ab}$ onto the whole of $\cX$ is used. However, the existence
%of such an extension is a rather deep fact, and for the time being we do not know how
%to generalize this to the Jordan setting.
\end{remark}

\section{From  Jordan pairs to Jordan geometries}\label{sec:PtoG}

The aim of this chapter is to construct a Jordan geometry starting from a Jordan pair $(V^+,V^-)$,
or from a Jordan algebra: 

\begin{theorem}\label{th:existence}
For every Jordan pair $(V^+,V^-)$ over $\K$, there is 
a Jordan geometry, having $(V^+,V^-)$ as associated Jordan pair.
More precisely, there is a functor from the category of Jordan pairs over $\K$ to Jordan geometries over $\K$
with base point. Under this functor, unital Jordan algebras correspond to Jordan geometries
with a pairwise transversal triple.
\end{theorem}

\begin{proof}
If $2$ is invertible in $\K$, then, as shown in \cite{Be02}, Th.\ 10.1, there is a generalized 
projective geometry with base point having $(V^+,V^-)$ as associated Jordan pair;
by Theorem \ref{th:gpg}, this geometry is a Jordan geometry, and thus the theorem is proved in this
case. 

If $2$ is not invertible in $\K$, we cannot use midpoints in order to define the inversions
$J^{xz}_a$, and hence we have to modify the construction: 
 the set $\cX$ and inversions of the type 
$J^{xx}_a = (-1)_x^a = (-1)_a^x$ are defined by the same methods as in \cite{Be02},
but  inversions of the type $J^{xz}_a$
for $x \not= z$ have to be defined in a different way: we define first the translation operators
$L^{xz}_a$, essentially by using a ``Jordan version'' of the exponential map for a Kantor-Koecher-Tits
algebra, and then let
\begin{equation}\label{JL-eqn}
J_a^{xz}:= L_a^{xz} J_a^{zz}.
\end{equation}
To be more specific, recall from \cite{Be02} or \cite{BeNe04} that a transversal pair
$(x,a)\in \cD_2$ corresponds to an {\em Euler operator}, i.e., to a 3-grading of the associated 
``Kantor-Koecher-Tits algebra'' $\g$ of the Jordan pair. 
The base point $(o,o')$ corresponds to the $3$-grading
$\g = V^+ \oplus \h \oplus V^-$, coming directly with the construction of $\g$.
Thus, given a transversal pair $(x,a)$, we may assume without loss of generality that
$(x,a)=(o,o')$ is the base point; then $U_a$ is naturally identified with $V^+$, and hence
 the condition $z \in U_a$ means that $z \in V^+$. 
 Defining the ``exponential'' $\exp(z) \in \Aut(\g)$ as in \cite{Lo95}, we then let
\begin{equation}
L_a^{xz}:= \exp(z) 
\end{equation}
(this depends on $(x,a)$ since $\exp(z)$ is defined with respect to a fixed $3$-grading), and define
$J_a^{xz}$ by (\ref{JL-eqn}).
Now one   has to prove that the Jordan structure map thus defined satisfies our axioms -- this proof is
quite lengthy, and essentially amounts to reverse the computations leading to the ``explicit formulae''
 given in the preceding section;
details are similar to the proof of \cite{Be02}, Th.\ 10.1, and will be omitted.
\end{proof}

%%%%%%%%%%%%

%%%

\appendix

\section{Inversive actions and symmetry actions}\label{App:SA}

In this appendix, we recall the definition of some  algebraic structures
(torsors, reflection spaces, symmetric spaces), and we define their ``actions'' on a
set. Since a group is defined by a binary law, there are just two kinds of actions
(left and right actions); a torsor is defined by a ternary law, and therefore we have three
kinds of actions: {\em left, right and middle}, or: {\em inversive  torsor actions}.

\subsection{Torsors}

\begin{definition}\label{def:torsor}
A {\em torsor} is a set $G$ with a map $G^3 \to G$, $(x,y,z) \mapsto (xyz)$ 
satisfying the following algebraic identities:

\begin{enumerate}
\item[(PA)]
{\em para-associative identity}:
 $((xuv)wz) = (x(wuv)z)=(xu(vwz)))$.
\item[(IP)]
 {\em idempotency identity}
 $(xxy)=y=(yxx)$.
\end{enumerate}

\nin
The {\em opposite torsor} is $G$ with $(xyz)^{opp}=(zyx)$, and
a torsor  is called {\em commutative} if $G = G^{opp}$, i.e.,  it satisfies the identity

\begin{enumerate}
\item[(C)] $(xyz)=(zyx)$.
\end{enumerate}

\nin
Categorial notions are defined in the obvious way.
In every torsor, {\em left-, right- and middle multiplication operators} are the maps
$G \to G$ defined by 
\begin{equation}
(xyz) =: m_{xz}(y) = \ell_{x,y}(z) = r_{z,y}(x) .
\end{equation}
\end{definition}

\nin
Every group $(G,e,\cdot)$ becomes a torsor by letting $(xyz)=xy\inv z$,
and every torsor is obtained in this way: thus torsors are ``groups with origin forgotten''.

\begin{lemma}\label{TorsorLemma}
In every torsor, the middle multiplication operators satisfy 

\begin{enumerate}
\item[(SA)]
$m_{xy} \circ m_{uv} \circ  m_{rs}  %= m_{(xvr),(suy)}
= m_{m_{xr} (v),m_{sy} (u)}$,
\item[(IP)] 
$m_{xz}(x)=z$, $m_{xz}(z)=x$.
\end{enumerate}
Conversely, a set $G$ with a map $m:G \times G \to \Bij(G)$, $(x,z) \mapsto m_{xz}$
satisfying (SA) and (IP), becomes a torsor by letting $(xyz):=m_{xz}(y)$.
The operator $m_{xz}$ is then invertible with inverse operator $m_{zx}$.
\end{lemma}

\begin{proof}
Applied to an element $z$, (SA) reads:
$(x(u(rzs)v)y)= ((xvr)z(suy))$. By direct check (easy if one uses the realization
$(xyz)=xy\inv z$), it is seen that this holds in any torsor.
Conversely, using (SA) and invoking (IP) twice, we get  (PA):

\nin $
(xy(uvw)) = m_{x,m_{uw}(v)}(y) = m_{m_{xy}(y), m_{uw}(v)}(y) 
=  m_{xw} m_{vy} m_{yu} (y) = (x(vuy)w)
$.
\end{proof}

\nin
Letting $u=y$ and $v=r$ in (SA), we get by (IP) the ``Chasles relation''

\ssk
\nin (SA')
$m_{xy} \circ m_{yv} \circ  m_{vs}  %= m_{(xvr),(suy)}
= m_{xs}$.

\ssk
\nin
It can be shown that, conversely, (SA') and (IP) imply (SA).

% Remark. In a semitorsor, $m_{xz}$ need not be invertible.
% Therefore we define the following concepts only for torsors.

\subsection{Inversive  torsor actions}

\begin{definition}\label{SymTorsorDef}
Let $(G,( -  -  -))$ be a torsor and $X$ a set.
An {\em inversive torsor action on $X$} is a map of $G \times G$ into the set of bijections of $X$
$$
G \times G \to \Bij(X), \qquad (x,z) \mapsto M_{xz}
$$
such that the following identities hold

\begin{enumerate}
\item[(STA1)]  $M_{xz} \circ M_{zx} = \id_X$
\item[(STA2)]
$M_{xz} \circ M_{uv} \circ M_{ab} = M_{(xva),(buz)}$
\end{enumerate}
\nin 
The inversive torsor action is called {\em commutative} if

\begin{enumerate}
\item[(CTA)] 
$M_{xz}=M_{zx}$
\end{enumerate}

\nin
(equivalently, if all $M_{xz}$ are of order two). 
According to the preceding lemma, every torsor has a natural inversive  action on itself, given
by $M_{xz}=m_{xz}$, which we call the {\em regular inversive  action (of $G$ on itself)}.
Spaces with $G$-inversive  action form a category in the obvious way, and {\em subspaces}
and {\em direct products} can be defined in this category. 
\end{definition}

\nin 
We interprete the operators $M_{xz}:X \to X$ as {\em generalized inverses}, whence the terminology. 
Letting $z=u$ and $v=a$ in (STA2), we get a {\em middle Chasles  relation}
\begin{equation}\label{eqn:MiddleChasles}
M_{xz} \circ M_{za} \circ M_{ab} = M_{xb} .
\end{equation}
However, it is not true that (\ref{eqn:MiddleChasles}) and (STA1) imply (STA2). 

\begin{remark}
These axioms have the following categorial interpretation: 
with the usual torsor structure $(fgh)=fg\inv h$ on $\Bij(X)$, (STA2) can
be rewritten in the form
$$
\bigl( M_{xz} M_{vu} M_{ab} \bigr)  = M_{(xva),(buz)}
$$
which means that $M$ can be interpreted as a torsor homomorphism
$$
M: G \times G^{opp} \to \Bij(X),\qquad (x,z) \mapsto M_{xz} \, .
$$
\end{remark}

\begin{remark}
It is not true that  an inversive action of a commutative torsor is always a commutative action.
For instance,
consider the following situation:
if $H$ is a subgroup of a group $G$,  then the regular action of $G$ on itself induces
an inversive action
$H \times H \to \Bij(G)$ given by
$M_{h,h'}(g)= hg\inv h'$.
This action is in general not commutative, even if $H$ as a group is commutative.
Indeed,  $(M_{xz})^2(u)= x z\inv u x\inv z$ may be a non-trivial map on $G$, although
it is trivial on $H$.
\end{remark}

\begin{remark}
Assume, in the preceding situation, that  $G$ is a compact Lie group and $H$ a 
maximal torus.  Then the elements $M_{x,x\inv}$ with $x^2 \in Z(G)$ (in particular, those
with $x^2 =e$) are of order $2$. On the other hand, when $x$ normalizes $H$, they 
stabilize $H$, and when $x$ centralizes $H$, they act trivially on $H$.
Taken together, we get the following interpretation of the Weyl group, together with
its set of generators of order two: it is the torsor of middle multiplication operators stabilizing $H$ and $e$,
generated by its involutive elements.  
%(consider, e.g., single one-parameter subgroups $H$ of $\Gl(n,\R)$).
% and what if we forget $e$ ? do we get the ``affine Weyl group'' ?
\end{remark}

\subsection{Left  and right torsor actions}

\begin{lemma}[Left  and right action] \label{L-Lemma}\label{LR-Lemma}
For an inversive  torsor action, 
the {\em left translation}  $L_{xv} \in \Bij(X)$,  defined by
$$
L_{xv}:= M_{xz} \circ M_{zv} ,
$$
depends only on $x,v \in G$, but not on the choice of $z$. We have   the identities
\begin{enumerate}
\item[(LTA1)] $L_{xx} = \id_X$,
\item[(LTA2)] $L_{xv}L_{uw} = L_{(xvu),w} = L_{x,(wuv)}$. 
\end{enumerate}
Similarly,
the {\em right translation}
 $R_{vx}:= M_{zv}\circ M_{xz}$ does not depend on $z$.
 Moreover,
for any inversive  torsor action,
left and right translations commute:
$$
L_{xv}\circ R_{yw}= R_{yw}\circ L_{xv} ,
$$
and if the symmetry action is commutative, then left- and right action agree:
$L_{xv}=R_{xv}$.
\end{lemma}

\begin{proof}
All claims are checked by direct computations. 
We show that $L_{xv}$ is well-defined: indeed, the equality  
$M_{xz} M_{zv} = M_{xw} M_{wv}$ is a direct consequence of (\ref{eqn:MiddleChasles}) and (STA1). 
In order to show that $L_{xv}$ and $R_{wu}$ commute, we may first reduce to the case
$x=w$, by observing that (LTA2) gives us  the {\em left Chasles relation}
\begin{equation}\label{eqn:Chasles}
L_{xv}\circ L_{vw}=L_{xw}.
\end{equation}
The relation $L_{xv} R_{xw}=R_{xw} L_{xv}$ is now proved by making appropriate choices
when expressing the $L$- and $R$-operators by two $M$-operators.
Finally, assuming that  the action is commutative, we get
$
L_{xv} \circ R_{vx} = M_{xx}M_{xv} M_{xv} M_{xx} = \id_X
$.
\end{proof}

\begin{corollary}[Transplantation formula]\label{la:Transp} Given
 a {\em  commutative} inversive  torsor action, we have, for all $x,o,z \in G$,
\[
M_{xz} = M_{xo}M_{oo}M_{zo} = M_{(xoz),o} = M_{m_{xz}(o),o}.
\]
\end{corollary}

\begin{proof}
$M_{xo}M_{oo}M_{zo}  = L_{xo} M_{zo} = L_{xo}M_{oz} = M_{xz}$
\end{proof}

\nin
There are some other algebraic identities valid for every inversive
 torsor action, such as the {\em intertwining relation between left and right actions}
\begin{equation}\label{eqn:Int}
M_{xx} \circ  L_{vx} \circ M_{xx} =  R_{(xvx),x} 
\end{equation}
which can also be written, for $x=e$ and $M_{ee}(g)=j(g)=g\inv$, and $L_g:=L_{g,e}$,
\begin{equation}\label{eqn:Int'}
j \circ L_g = R_{g\inv} \circ j .
\end{equation}
 Note also that identities for $R$-operators correspond to identities for $L$-operators,
with reversed composition in $\Bij(X)$ and reversed order of indices. 

\begin{definition}
A {\em left torsor action} of a torsor $G$ on a set $X$ is a map
$$
G \times G \to \Bij(X), \qquad (x,y) \mapsto L_{xy}
$$
(if there is risk of confusion we write also $L_{x,y}$ instead of $L_{xy}$)
such that the identities (LTA1) and (LTA2) from the preceding lemma hold. 
\nin {\em Right actions} are defined similarly.
The {\em regular left (right) action of $G$ on itself} is defined by the lemma.
\end{definition}

\begin{lemma}
Let $G$ be a group with neutral element $e$ and its usual torsor structure $(xyz)=xy\inv z$. 
Then we have an equivalence of categories between left group actions of $G$ and
left torsor actions of $G$.
\end{lemma}

\begin{proof}
Given a left group action $G \times X \to X$, $(g,x) \mapsto L_g(x)$, we let 
$$
L_{x,y}:= L_x (L_y)\inv .
$$
Conversely, given a left torsor action, let $L_g := L_{g,e}$, and the claim
follows by a straightforward check of definitions.
\end{proof}

\subsection{On the structure of inversive actions}
The preceding two lemmas say that left  and right actions of torsors are nothing new, compared
to usual group actions, whereas inversive  actions are commutig left  and right actions
together with some operator $j$ satisfying the intertwining relation (\ref{eqn:Int'}).
One may check that, conversely, if we have commuting left and right actions of a group $G$
on a set $X$, together with a map of order two $j:X\to X$ satisfying the intertwining relation,
we can reconstruct an inversive  torsor action.

Motivated by this obervation, 
one will look at the behaviour of $G \times G$-orbits $\mathcal O$ under $j$.
If $j ({\mathcal O}) \cap {\mathcal O}$ is empty, then $j$ is equivalent to the exchange map
between two copies of this orbit, exchanging ``left'' and ``right''.
In the other case, one will have to distinguish whether $j$ has a fixed point in $\mathcal O$, or not.
If there is a fixed point $p$, the stabilizer $H$ of $p$ must be a normal subgroup, and we
get a version of the regular symmetry action on the quotient group $G/H$.
The remaining case, where $j$ has no fixed point in $\mathcal O$, seems to be more
difficult to analyze. 

%[try to find at least one example !
%in other words: are there group-like structures, with ``inversion having no fixed point'' ?]

\begin{comment}%%%not so useful....
\begin{remark}
Given a symmetry torsor action of $G$ on $X$, we may define
$(xaz):= M_{xz}(a)$ and $(xya):=L_{xy}(a)$ and $(ayz):=R_{yz}(a)$;
then the para-associative law holds, if we pay attention to what sets elements belong.
(see philosophy of general representations!)
We might define semitorsor actions in the same way... (but won't do it...)
For usual left actions, the element $a$ has always to be at the end of the sign chain.
\end{remark}
\end{comment}%%%%%%%%

\subsection{Reflection spaces and symmetric spaces}

\begin{definition} \label{ReflectionDef}
A {\em  reflection space} is a set together with a map
$s: M \to \Bij(M)$, $x \mapsto s_x$ such that %(using the notation $\mu(x,y)=s_x(y)=(s(x))(y)$)
the following identities hold:

\begin{enumerate}
\item[(R1)]  (idempotency) $s_x(x)=x$,
\item[(R2)] (inversivity)  $s_x \circ s_x = \id_M$,
\item[(R3)] (distributivity)  $s_x s_z s_x = s_{s_x (z)}$.
\end{enumerate}

\nin
Reflection spaces form a category. The subgroup $G(M)$ of $\Aut(M,\mu)$ generated by
all $s_x s_y$ with $(x,y) \in M^2$ is called the {\em transvection group} of $M$.
If, moreover, $M$ is a Weil  manifold (cf.\ subsection \ref{ssec:Weils}), then $M$ is called a {\em symmetric space} if,
 for
every $x \in M$, the tangent map $T_x (s_s)$ of $s_x$ at its fixed point $x$ is equal to
$ - \id_{T_x M}$.
\end{definition}

\subsection{Symmetry action of a reflection space}

\begin{definition}\label{RefActionDef}
Let $M$ be a reflection space and $X$ a set.
A {\em symmetry action of $M$ on $X$} is a map
$
M \to \Bij(X)$, $x \mapsto S_x
$
such that 
\begin{enumerate}
\item[(S1)]
$S_x \circ S_x = \id_X$,
\item[(S2)]
$S_x S_y S_x = S_{s_x(y)}$.
\end{enumerate}
For $X=M$, we have a symmetry action of $M$ on itself given by $S_x=s_x$, which
we call the {\em regular symmetry action (of $M$ on itself)}.
As above, categorial notions are defined. 
\end{definition}

\begin{comment}%%%%%%%If we assume that $S_x$ is bijective, then (R2) follows:
\begin{lemma}
For every symmetry action and every $x \in M$,
the bijections $S_x$ are of order two: $S_x S_x = \id_X$.
In particular, the maps $s_x$ are of order two. 
\end{lemma}

\begin{proof}
For $x=y$, (R2) implies 
$(S_x)^3 = S_x$, and since $S_x$ is a bijection, it follows that $S_x S_x = \id_X$.
\end{proof}
\end{comment}%%%%%%%%%%%%%%%%%%%%%%%%%%%%%%%

\begin{lemma}
If $(x,z) \mapsto M_{xz}$ is an inversive  action of a torsor $G$, then we get 
a symmetry action of $G$, seen as reflection space, by
$G \to \Bij(X)$,
$x \mapsto S_x:=M_{xx}$.
\end{lemma}

\begin{proof}
$S_x S_y S_x = M_{xx} M_{yy}M_{xx}=M_{(xyx),(xyx)} = M_{s_x(y),x_x(y)} = S_{s_x(y)}$
\end{proof}

In general, left or right actions of $G$ do not give rise to symmetry actions of the symmetric space
$G$;
and in general, symmetry actions of reflection spaces do not give rise to actions
of the group $G(M)$ (cf.\ remarks in \cite{Be00}: already on the infinitesimal level this does not
hold since the standard imbedding of a Lie triple system is in general not functorial).

\begin{definition}\label{def:Transv}
Given a symmetry action $M \to \Bij(X)$,
$x \mapsto S_x$, we define the {\em transvection operators} by
$
Q_{xy}:=S_x S_y  \in \Bij(X) .
$
\end{definition}

These operators share some properties with the translation operators of left or right torsor actions:
we have an analog of the Chasles relation (\ref{eqn:Chasles})
$Q_{xy}Q_{yz} = Q_{xz}$, and $Q_{xx}=\id_X$, whence $(Q_{xy})\inv = Q_{yx}$,
but in contrast to left and right translations, 
the composition of two transvections is in general no longer a transvection. 
Instead, we have the {\em fundamental formula}
\begin{equation}\label{eqn:Fu}
Q_{xy} Q_{zy} Q_{xy} = S_x S_y S_z S_y S_x S_y = S_{S_x S_y (z)} S_y =
Q_{Q_{xy},y} \, .
\end{equation}


\begin{thebibliography}{BeNe05}

\bibitem[Ba73]{Ba73}
F.\ Bachmann,
{\em Aufbau der Geometrie aus dem Spiegelungsbegriff},
Springer Grundlehren Band 96, Springer, Berlin 1973.

%\bibitem[BB]{BB}
%Bertelson Bieliavsky arxiv

\bibitem[Be00]{Be00}
W.\ Bertram, {\it The geometry of Jordan- and Lie structures},
Springer LNM 1754, Springer, Berlin 2000.

\bibitem[Be02]{Be02}
W.~Bertram,
Generalized projective geometries: general theory and equivalence with Jordan structures,
\textit{Adv. Geom.} \textbf{2} (2002), 329--369 (electronic version: preprint 90 on Jordan preprint server
\url{http://molle.fernuni-hagen.de/~loos/jordan/index.html}).

\bibitem[Be03]{Be03}
W.~Bertram,
The geometry of null systems, Jordan algebras and von Staudt's Theorem,
\textit{Ann. Inst. Fourier} \textbf{53} (2003) fasc. 1, 193--225 (preprint 113, Jordan server).

\bibitem[Be04]{Be04}
W.~Bertram, From linear algebra via affine algebra to projective algebra,
\textit{Linear Algebra and its Applications} \textbf{378} (2004), 109--134 (preprint 89, Jordan server).

%\bibitem[Be07]{Be07}
%W.~Bertram,
%Jordan structures and non-associative geometry.
%In \textit{Trends and Developments in Infinite Dimensional Lie Theory},
%Progress in Math., Birkhaeuser, 2008, to appear;
%arXiv: \url{math.RA/0706.1406}.



\bibitem[Be08]{Be08}
W.~Bertram,
Differential Geometry, Lie Groups and Symmetric Spaces over General Base Fields and Rings. 
  Mem.\  AMS 192, no.900 (2008),                         
       arXiv \url{http://arxiv.org/abs/math/0502168}.

\bibitem[Be08b]{Be08b}
W.~Bertram,
Homotopes and conformal deformations of symmetric spaces.
\textit{J. Lie Theory} \textbf{18} (2008), 301--333;
arXiv \url{http://arxiv.org/abs/math.RA/0606449}.

%\bibitem[Be13]{Be13}
%W.\ Bertram,
 %Simplicial differential calculus, divided differences, and construction of Weil functors.        
    %   Forum Mathematicum 25 (1) (2013), 19-47.                   arxiv 
      %  \url{http://arxiv.org/abs/1009.2354}.

\bibitem[Be14]{Be14}
W.\ Bertram, Weil spaces and Weil Lie groups, preprint, 
\url{http://arxiv.org/abs/1402.2619}

\bibitem[BeKi09a]{BeKi1}
W.~Bertram and M.~Kinyon,
Associative Geometries. I: Torsors, Linear Relations and Grassmannians,
 J.\  Lie Theory {\bf 20} (2) (2010), 215-252.    
arXiv \url{http://arxiv.org/abs/0903.5441}.


\bibitem[BeKi09b]{BeKi2}
W.~Bertram and M.~Kinyon,
Associative Geometries. II: Involutions and classical groups,
J.\ Lie Theory {\bf 20} (2) (2010), 253-282.
arXiv \url{http://arxiv.org/abs/0909.4438}.

\bibitem[BeKi12]{BeKi12}
W.~Bertram and M.~Kinyon,
Torsors and ternary Moufang loops arising in projective geometry,
arxiv \url{http://arxiv.org/abs/math/1206.2222}.


\bibitem[BeL08]{BeL08}
W.~Bertram and H.~Loewe,
Inner ideals and intrinsic subspaces,
\textit{Adv. in Geometry} \textbf{8} (2008), 53--85;
arXiv \url{http://arxiv.org/abs/math/0606448}.

\bibitem[BeNe04]{BeNe04}
W.~Bertram and K.-H.~Neeb,
Projective completions of Jordan pairs. Part I: The generalized projective geometry of a Lie algebra, 
 J. of Algebra 227 , 2 (2004), 474--519;
       arXiv \url{http://arxiv.org/abs/math/0306272}.


\bibitem[BeNe05]{BeNe05}
W.~Bertram and K.-H.~Neeb,
Projective completions of Jordan pairs. II: Manifold structures and symmetric spaces,
\textit{Geom. Dedicata} \textbf{112} (2005), 73 -- 113;
arXiv \url{http://arxiv.org/abs/math/0401236}.

\bibitem[BeS11]{BeS11}
W.\ Bertram and A.\ Souvay,
A general approach to Weil functors,
arxiv \url{http://arxiv.org/abs/1111.2463}.


\bibitem[Bue]{Bue}
Buekenhout, F. (ed.), 
{\em Handbook of Incidence Geometry -- Buildings and Foundations}, Elsevier, 1995.


%\bibitem[Cer43]{Cer43}
%J.~Certaine,
%The ternary operation $(abc)=ab\sp {-1}c$ of a group,
%\textit{Bull. Amer. Math. Soc.} \textbf{49} (1943), 869--877.


\bibitem[KMS93]{KMS}
I.\ Kolar, P.\ Michor and J.\  Slovak, {\em Natural Operations in Differential Geometry}, Springer 1993.

\bibitem[Lo67]{Lo67}
O.\ Loos, Spiegelungsr\"aume und homogene symmetrische R\"aume, Math.\ Z.\ {\bf 99} (1976), 
141 -- 170. 


\bibitem[Lo69]{Lo69}
O.~Loos,
\textit{Symmetric Spaces I},
Benjamin, New York, 1969.


\bibitem[Lo75]{Lo75}
O.~Loos,
\textit{Jordan Pairs},
Lecture Notes in Math. \textbf{460}, Springer, New York, 1975.

\bibitem[Lo79]{Lo79}
O.~Loos,
On algebraic groups defined by Jordan pairs, Nagoya math.\ J.\ {\bf 74} (1979), 23 -- 66. 

\bibitem[Lo95]{Lo95}
O.~Loos,
Elementary groups and stability for Jordan pairs,
$K$-Theory {\bf 9} (1995), 77--116.

\bibitem[Sp73]{Sp}
T.\ Springer,
{\it Jordan algebras and algebraic groups}, 
Springer-Verlag, Berlin 1973.

\bibitem[St67]{St67}
R.\ Steinberg,
{\it Lectures on Chevelley Groups}, Yale University, 1967.

\bibitem[Wi81]{Wi}
J.B.\ Wilker, 
Inversive Geometry, p. 379--442 
in: {\it The Geometric Vein} - Coxeter Festschrift (editors Davis et al.), Springer 1981.

\end{thebibliography}
\end{document}